\documentclass[11pt]{amsart}
\usepackage{amsmath}
\usepackage{amssymb}
\usepackage{amsthm}
\usepackage{amscd}
\usepackage{fontenc}
\usepackage{url}
\usepackage{hyperref}
\usepackage[all]{xy}
\usepackage{graphicx}
\usepackage{tikz}

\numberwithin{equation}{section}

\newcounter{AbcT}

\newtheorem {Theorem}    {Theorem}[section]
\newtheorem* {Theorem1.9}    {Theorem 1.9}
\newtheorem* {Theorem1.6}    {Theorem 1.6}
\newtheorem* {Question1.7}    {Question 1.7}

\newtheorem* {Remark}	{\bf{Remark}}

\newtheorem* {Hypothesis}    {Regularity Hypothesis}

\newtheorem {Lemma}      [Theorem]    {Lemma}
\newtheorem {Corollary}   [Theorem] {Corollary}
\newtheorem {Proposition}[Theorem]    {Proposition}
\newtheorem {Claim}      [Theorem]    {Claim}
\newtheorem {Observation}[Theorem]    {Observation}

\theoremstyle{remark}

\theoremstyle{definition}
\newtheorem {Definition}[Theorem] {Definition} %[section]

\newcounter{DM@bibnum}

\newcommand {\F} {{\mathcal F}}

\newcommand {\K} {K}

\newcommand{\la}{\langle}
\newcommand{\ra}{\rangle}

\def\IA{{\rm IA}}

\def\GL{{\rm GL}}
\def\Sp{{\rm Sp}}

\def\supp{{supp\,}}

\def\Aut{{\rm Aut}}
\def\Inn{{\rm Inn}}
\def\Out{{\rm Out}}
\def\Hom{{\rm Hom}}

\def\Mod{{\rm Mod}}

\def\Ker{{\rm Ker\,}}
\def\Im{{\rm Im\,}}

\def\NSL_2{{\mathcal N SL_2}}

\def\eps{\varepsilon}

\def\lam{\lambda}            
                
\def\phi{\varphi}

\def\calB{{\mathcal B}}

\def\calI{{\mathcal I}}

\def\calK{{\mathcal K}}

\def\calM{{\mathcal M}}

\def\calT{{\mathcal T}}

\def\hbar{\bar h}

\def\dbF{{\mathbb F}}

\def\dbN{{\mathbb N}}

\def\dbQ{{\mathbb Q}}
\def\dbR{{\mathbb R}}
\def\dbS{{\mathbb S}}

\def\dbZ{{\mathbb Z}}

\def\skv{{\vskip .12cm}}

\begin{document}

\title{Effective finite generation for $[\IA_n,\IA_n]$ and the Johnson kernel}

\author{Mikhail Ershov}
\address{University of Virginia}
\email{ershov@virginia.edu}
\author{Daniel Franz}
\address{Jacksonville University}
\email{dfranz1@ju.edu}

\begin{abstract}
Let $\IA_n$ denote the group of $\IA$-automorphisms of a free group of rank $n$, and let $\mathcal I_n^b$ denote the Torelli subgroup of
the mapping class group of an orientable surface of genus $n$ with $b$ boundary components, $b=0,1$. In 1935 Magnus proved that
$\IA_n$ is finitely generated for all $n$, and in 1983 Johnson proved that $\mathcal I_n^b$ is finitely generated for $n\geq 3$.

It was recently shown that for each $k\in\dbN$, the $k^{\rm th}$ terms of the lower central series $\gamma_k \IA_n$ and
$\gamma_k\mathcal I_n^b$  are finitely generated when $n>>k$; however, no information about finite generating sets was known for $k>1$. The main goal of this paper is to construct an explicit finite generating set for $\gamma_2 \IA_n = [\IA_n,\IA_n]$ and almost explicit finite generating sets for $\gamma_2\mathcal I_n^b$ and the Johnson kernel, which contains $\gamma_2\mathcal I_n^b$ as a finite index subgroup.
\end{abstract}

\maketitle
\section{Introduction}

\subsection{Discussion of the problem}
\label{sec:introresults}

Given non-negative integers $n$ and $b$, let $\Sigma_n^b$ be an orientable surface of genus $n$ with $b$ boundary components, and let $\Mod_n^b=\Mod(\Sigma_n^b)$ be its mapping class group. The corresponding Torelli group $\calI_n^b$
is the subgroup of $\Mod_n^b$ consisting of elements acting trivially
on $H_1(\Sigma_n^b,\dbZ)$. In this paper we will only consider the cases $b=0,1$. Given $n\in\dbN$, let
$F_n$ denote a free group on $n$ generators, and let $\IA_n$ be the subgroup of $\Aut(F_n)$ consisting of automorphisms acting trivially on the abelianization $F_n^{ab}=F_n/[F_n,F_n]$. The group $\IA_n$ is often called the Torelli subgroup
of $\Aut(F_n)$ and is known to behave similarly to $\calI_n^1$ in many ways.
\vskip .12cm

In 1935, Magnus~\cite{Ma} proved that $\IA_n$ is finitely generated for all $n\geq 2$; in fact, he found an explicit and simple-to-describe generating set of smallest possible cardinality $n{n\choose 2}$. In 1983, Johnson~\cite{Jo:fg} proved that $\calI_n^b$ is finitely generated for $n\geq 3$; his generating set is also explicit and also of optimal size for $n=3$, but of considerably larger size in general (growing exponentially in $n$). More recently, Putman~\cite{Pu1} found a smaller generating set whose size grows cubically with $n$, which is known to be asymptotically optimal.\footnote{The smallest size of a generating set for $\IA_n$ is indeed $n{n\choose 2}$ since its abelianization $\IA_n^{\rm ab}$
is free of rank $n{n\choose 2}$ -- see \S~3 for details. Likewise the asymptotic optimality of the generating set from Putman~\cite{Pu1}
follows from the fact that the torsion-free rank of $(\calI_n^b)^{\rm ab}$ is cubic in $n$ which was proved by Johnson~\cite{Jo:ab1,Jo:ab2} -- see \S~4 for details.}

\vskip .12cm

It was a very interesting question whether the commutator subgroup $[G,G]$ is finitely generated for $G=\IA_n$ or $\calI_n^b$. In both cases $G$ has large abelianization, so there was no a priori reason to expect $[G,G]$ to be finitely generated. On the other hand, $G$ possesses a generating set in which many pairs of generators commute, which can be seen as positive evidence for finite generation of $[G,G]$. An additional motivation for the finite generation question in the mapping class group case is given by the fact that 
$[\calI_n^b,\calI_n^b]$ is a finite index subgroup of the Johnson kernel $\calK_n^b$, a group of major interest in topology.
\vskip .12cm

In \cite{EH} it was proved that the commutator subgroups of the Torelli groups are indeed finitely generated in sufficiently large rank: $[\IA_n,\IA_n]$ for $n\geq 4$ and $[\calI_n^b,\calI_n^b]$ for $n\geq 12$. 
In \cite{CEP}, $[\calI_n^b,\calI_n^b]$ was shown to be finitely generated for all $n\geq 4$, and finite generation
was also extended to some higher terms of the lower central series, namely
$\gamma_k \IA_n$ for $n\geq 4k-3$ and $\gamma_k \calI_n^b$ for $n\geq 2k+1$ (see \S~\ref{sec:introQR} for an additional discussion).
However, neither \cite{EH} nor \cite{CEP} dealt directly with the finite generation question, instead reducing the problem to an analysis of BNS invariants.
\vskip .12cm

Given a finitely generated group $G$, its character sphere $\dbS(G)$ is the set of nonzero homomorphisms from $G$
to $(\mathbb R,+)$ modulo the equivalence given by multiplication by positive scalars. The BNS-invariant of $G$,
introduced by Bieri, Neumann and Strebel in \cite{BNS} and denoted by $\Sigma(G)$, is a subset of $\dbS(G)$ which determines which subgroups of $G$ containing $[G,G]$ are finitely generated (see Theorem~\ref{thm:BNS}, often called the BNS criterion). In particular, $[G,G]$ itself is finitely generated if and only if $\Sigma(G)=\dbS(G)$, and it is the latter equality that was established in \cite{EH} and \cite{CEP} for $G=\IA_n$ and $\calI_n^b$ for $n\geq 4$. Finite generation of higher terms of the lower central series was established in \cite{CEP} by an inductive application of the BNS criterion. 
\vskip .12cm

Since the proof of the BNS criterion is not effective, \cite{EH} and \cite{CEP} did not yield an actual construction of finite generating sets for $[G,G]$ (or higher terms) for $G=\IA_n$ and $\calI_n^b$.
The main goal of the present paper is to give an effective proof of finite generation for $[G,G]$ when
$n\geq 8$. In \S~\ref{sec:introQR} we will discuss the main obstacle to extending this method to $\gamma_k G$ for
$k>2$.
\vskip .12cm

The proof of finite generation that we will provide does not make a formal reference to the BNS invariant; however, it relies on the proof of the BNS criterion in a substantial way. Essentially, we follow the proof of the BNS criterion given in \cite{Str} (which is considerably simpler than the original argument from \cite{BNS}), replace all the non-effective steps with explicit constructions and make some simplifications which are not possible in general. In addition to providing explicit generating sets for $[G,G]$, the proof in this paper is algorithmic in the following sense -- given a sufficiently nice generating set $S$ of $G$ and an element $g\in [G,G]$ expressed in terms of $S$, our proof yields a procedure for writing $g$ in terms of a finite generating set for $[G,G]$ (which is explicitly derived from $S$). Our general method for proving effective finite generation of $[G,G]$ that will be developed in \S~\ref{sec:EFG}
is sufficiently flexible and could be applicable to other groups.
\subsection{Main results}

We proceed with stating our main result for $\IA_n$. Throughout the paper, for a group $G$ and elements $x,y\in G$, we set $x^y=y^{-1}xy$ and $[x,y]=x^{-1}y^{-1}xy$.

\begin{Theorem}
\label{thm:main_IAn} Let $n\geq 8$. Let $N=n{n\choose 2}$, and let $S=\{s_1,\ldots, s_N\}$ be the standard generating set for
$\IA_n$ constructed by Magnus (see the beginning of \S~\ref{sec:IAn}). Then $[\IA_n,\IA_n]$ is generated by elements of the form
$$[s_i, s_j]^{s_i^{a_i}s_{i+1}^{a_{i+1}}\ldots s_N^{a_N}}
\mbox{ where } 1\leq i<j\leq N \mbox{ and } 0\leq |a_m|< 5\cdot 10^{12} \mbox{ for each } m.$$ 
In particular, the minimal number of generators of $[\IA_n,\IA_n]$ is at most $n{n\choose 2}\cdot ({10}^{13})^{n{n\choose 2}}$.
\end{Theorem}
\begin{Remark}\rm (1) It was previously known~\cite{To}
\footnote{It suffices to prove this result when $G$ is free, in which case the given generating set for $[G,G]$
freely generates $[G,G]$. The latter is the main result of \cite{To}; recently a more geometric proof was given in \cite{Pu2}.} 
that if $G$ is any group generated by a finite set
$\{x_1,\ldots, x_k\}$, then $[G,G]$ is generated by elements of the form $[x_i, x_j]^{x_i^{a_i}x_{i+1}^{a_{i+1}}\ldots x_k^{a_k}}$ with $1\leq i<j\leq k$ and $a_m\in\dbZ$. Thus, the novel part of Theorem~\ref{thm:main_IAn} is that in the case $G=\IA_n$ it suffices to take only such elements where $a_m$ are bounded by an explicit constant independent of $n$.

(2) In \S~3 we will show that the number of generators of $[\IA_n,\IA_n]$ is actually bounded by a function of the form $C^{n^2}$ (see Theorem~\ref{thm:main_IAn2}) which is slightly better than a bound of the form $C^{n^3}$ given by
Theorem~\ref{thm:main_IAn}.

(3) Our method of proof is in principle applicable to all $n\geq 4$, but would yield a constant larger than $5\cdot 10^{12}$
for $n=6,7$ and a much larger constant for $n=4,5$.
\end{Remark}

We now turn to the mapping class groups. In this case we will describe generating sets for two different subgroups of $\calI_n^b$ -- the commutator subgroup $[\calI_n^b,\calI_n^b]$ and the Johnson kernel $\calK_n^b$. 
One can define the Johnson kernel algebraically, as the second term of the Johnson filtration of $\calI_n^b$ 
(see \S~\ref{sec:prelim} for the definition of the Johnson filtration in the case $b=1$; the definition in the case
$b=0$ is similar) or topologically, as the subgroup of $\Mod_n^b$ generated by Dehn twists about separating curves. It also follows from work of Johnson that $\calI_n^b/\calK_n^b$ is the largest quotient of $\calI_n^b$ which is  abelian and torsion-free. This, together with finite generation of $\calI_n^b$, immediately implies that $\calK_n^b$ contains $[\calI_n^b,\calI_n^b]$ as a finite-index subgroup for $n\geq 3$.

\vskip .12cm

For simplicity we state Theorems~\ref{thm:main_Torelli}~and~\ref{thm:main_Torelli2} below for $b=1$. 
It is well known that there is a natural surjective map
$\Mod_n^1\to \Mod_n^0$ which sends $\calI_n^1$ to $\calI_n^0$ and $\calK_n^1$ to $\calK_n^0$, so any finite generating set
of $\calK_n^1$ yields the corresponding finite generating set for $\calK_n^0$.

\begin{Theorem}
\label{thm:main_Torelli} Let $n\geq 4$, let $G=\calI_n^1$ and $K=\calK_n^1$. Let $N={2n\choose 3}$ and
$M={2n\choose 2}+{2n\choose 1}+{2n\choose 0}$.
The following hold:
\begin{itemize}
\item[(1)] $G$ has a generating set $S^{(1)}\sqcup S^{(2)}\sqcup S^{(3)}$ with the following properties: 
\begin{itemize}
\item[(i)] $|S^{(1)}|=N$ and the elements $s_1,\ldots, s_N$ of $S^{(1)}$ project to a basis of $G/K$ (which is free abelian of rank $N$).
\item[(ii)] $|S^{(2)}|=M$, the elements $t_1,\ldots, t_M$ of $S^{(2)}$ lie in $K$, and $S^{(2)}$ projects to a basis of $K/[G,G]$ (which is a vector space over $\dbF_2$ of dimension $M$).
\item[(iii)] $S^{(3)}$ is contained in $[G,G]$ and $|S^{(3)}|=42{n\choose 3}-|S^{(1)}|-|S^{(2)}|$.
\end{itemize}
\vskip .12cm
\item[(2)] There exists an absolute constant $R$ such that $K$ is generated by elements of the form
\begin{itemize}
 \item[(a)] $[s_i,s_j]^{s_i^{a_i}s_{i+1}^{a_{i+1}}\ldots\,\, s_N^{a_N}}$ where $1\leq i<j\leq N$ and $|a_m|\leq R$ for all $m$;
 \item[(b)] $x^{s_1^{a_1}s_{2}^{a_{2}}\ldots s_N^{a_N}}$ where $x\in S^{(2)}\cup S^{(3)}$ and $|a_m|\leq R$ for all $m$.
\end{itemize}
\vskip .12cm
\item[(3)] $[G,G]$ is generated by elements of the form
\begin{itemize}
 \item[(a)] $[s_i,s_j]^{s_i^{a_i}s_{i+1}^{a_{i+1}}\ldots\,\, s_N^{a_N}t_1^{\eps_1}\ldots t_M^{\eps_M}}$ with $1\leq i<j\leq N$; 
 \item[(b)] $[s_i,t_j]^{s_i^{a_i}s_{i+1}^{a_{i+1}}\ldots\,\, s_N^{a_N}t_1^{\eps_1}\ldots t_M^{\eps_M}}$ with $1\leq i\leq N$ and $1\leq j\leq M$;
 \item[(c)] $[t_i,t_j]^{t_i^{\eps_i}t_{i+1}^{\eps_{i+1}}\ldots\,\, t_M^{\eps_M}}$ with $1\leq i<j\leq M$; 
 \item[(d)] $x^{s_1^{a_1}s_{2}^{a_{2}}\ldots s_N^{a_N}t_1^{\eps_1}\ldots t_M^{\eps_M}}$ with $x\in S^{(3)}$,
\end{itemize}
where in each part $|a_m|\leq R$, $\eps_m\in\{0,1\}$ for all $m$ and $R$ is the same as in (2).
\end{itemize}
\end{Theorem}
\begin{Remark}\rm We will construct explicit generating sets for $[\calI_n^b,\calI_n^b]$ and $\calK_n^b$ only for $n\geq 8$.
Theorem~\ref{thm:main_Torelli} is still valid for all $n\geq 4$ since $[\calI_n^b,\calI_n^b]$ and $\calK_n^b$ are known to be finitely generated for all $n\geq 4$ by \cite{CEP} and we are not making any assertions about the constant $R$ above.

The problem of explicitly estimating $R$, at least for $n\geq 8$, reduces to a certain computation in the Torelli group 
$\calI_5^1$. We did not compute a precise upper bound, but we believe that this can be achieved by carefully examining the proofs of Johnson~\cite{Jo:fg} and Stylianakis~\cite{Sty} (see the end of \S~\ref{sec:pfTorelli2} for a detailed discussion). 
\end{Remark}

\subsection{Generating the Johnson kernel by finitely many Dehn twists}

One drawback of the generating sets from Theorem~\ref{thm:main_Torelli} is that they do not seem to have any natural geometric or topological interpretation. We will now address this issue in the case of the Johnson kernel $\calK_n^1$. Recall that $\calK_n^1$ is generated by the Dehn twists about separating curves. Thus, one way to produce an explicit geometrically meaningful finite generating set for $\calK_n^1$ is to show that
$\calK_n^1$ is generated by the Dehn twists about separating curves of explicitly bounded word length.

Let us now make our task more precise. Fix a point $p_0$ on the boundary of $\Sigma= \Sigma_n^1$. The fundamental group $\pi_1(\Sigma,p_0)$ is free of rank $2n$ and admits a basis ${\alpha_1,\beta_1\ldots, \alpha_{n},\beta_n}$ such that 
$\prod_{i=1}^n [\alpha_{i},\beta_{i}]$ is represented
by $\partial \Sigma$; below we will refer to such a basis as {\it natural}. Fix a natural basis $\omega$ of $\pi_1(\Sigma,p_0)$.
Given $m\in\dbN$, let $SC(m)$ be the set of all elements of $\pi_1(\Sigma,p_0)$ which have word length at most $m$ with respect to $\omega$ and which are represented by a separating simple curve on $\Sigma$. Let 
$T_{sc}(m)\subset \Mod(\Sigma_n^1)$ be the set of Dehn twists about the elements of $SC(m)$. As we will explain in \S~4, any two 
natural bases of $\pi_1(\Sigma,p_0)$ lie in the same orbit under the action of $\Mod(\Sigma,p_0)$ on $\pi_1(\Sigma,p_0)$ (see the remark after Theorem~\ref{thm:Zieschang}). This easily 
implies that $T_{sc}(m)$  is independent of $\omega$ up to conjugation; in particular, the smallest $m$ for which  $T_{sc}(m)$ 
generates $\calK_n^1$ does not depend on the choice of $\omega$. 

We can now formulate our theorem describing an explicit finite generating set for $\calK_n^1$ consisting of Dehn twists:

\begin{Theorem}
\label{thm:main_Torelli2}
Assume that $n\geq 4$. There exists an absolute constant $D$ such that $\calK_n^1$ is generated by the set $T_{sc}(D^{n^3})$ defined above.  
\end{Theorem}

Theorem~\ref{thm:main_Torelli2} will be obtained as a relatively easy consequence of Theorem~\ref{thm:main_Torelli} and some auxiliary results established in \S~4. The constant $D$ in Theorem~\ref{thm:main_Torelli2} can be expressed in terms of the constant $R$ from Theorem~\ref{thm:main_Torelli} and two other absolute constants which we believe can be estimated explicitly for $n\geq 8$.

\subsection{Some questions and remarks}
\label{sec:introQR}

As we already stated in \S~\ref{sec:introresults}, finite generation results from \cite{EH} were extended to higher terms of the lower central series in \cite{CEP}, where it was shown that $\gamma_k \IA_n$ is finitely generated whenever $n\geq 4k-3$ and $\gamma_k \calI_n^b$ is finitely generated
whenever $n\geq 2k+1$. Thus it is natural to ask if the method of the current paper can also provide explicit generating sets for higher terms.
We did not succeed in doing this.

The proof in \cite{CEP} was ineffective for two reasons: similarly to \cite{EH}, it exploited the BNS invariant. In addition, a combinatorial calculation from \cite{EH} was replaced by an ineffective Zariski density argument in 
\cite{CEP}. The latter is not a real obstacle to constructing explicit generating sets, and one can show that algebraic geometry can be eliminated from the proof in \cite{CEP} at the expense of increasing the lower bound on $n$ in terms of $k$ (for which we are claiming that $\gamma_k G$ is finitely generated), with the new bound being quadratic in $k$. What does cause a problem is the fact that for $k\geq 2$, $G=\IA_n$ or $\calI_n^b$,
very little seems to be known about the torsion in $\gamma_k G/\gamma_{k+1}G$ or the presentation of $\gamma_k G/\gamma_{k+1}G$ by generators and relations (as an abelian group). 
\vskip .2cm
We conclude this section with some speculations on the asymptotic growth of the number of generators of 
$[\IA_n,\IA_n]$ and $\calK_n^b$ as $n\to\infty$. Below, for a group $\Gamma$ we will denote by $d(\Gamma)$ the minimal number of generators of $\Gamma$. The following inequalities are obvious:
$$d(\Gamma)\geq d(\Gamma^{\rm ab})=d(H_1(\Gamma,\dbZ))\geq \dim H_1(\Gamma,\dbQ).$$
It is known that $\dim H_1(\calK_n^b,\dbQ)$ grows polynomially with $n$, and in fact
a precise formula for this dimension for $n\geq 6$ can be immediately extracted from Theorem~1.4 in a recent paper
of Morita, Sakasai and Suzuki~\cite{MSS} which, in turn, makes essential use of an earlier work of 
Dimca, Hain and Papadima~\cite{DHP}. This provides at least some evidence that $d(\calK_n^b)$ might grow polynomially as well. We are not aware of analogous results dealing with $[\IA_n,\IA_n]$.
\vskip .2cm
Finally, it is natural to ask if the assertions of Theorem~\ref{thm:main_IAn} and Theorem~\ref{thm:main_Torelli}(2)
would remain true if the condition $|a_m|\leq C$ on the exponents is replaced by the much more restrictive condition $\sum |a_m|\leq C$ for some absolute constant $C$. Clearly, if this stronger version of Theorem~\ref{thm:main_IAn} (resp. Theorem~\ref{thm:main_Torelli}(2)) holds, it would immediately imply polynomial growth for $d([\IA_n,\IA_n])$ (respectively, $d(\calK_n^b)$).

\vskip .2cm
{\bf Acknowledgments}. We are extremely grateful to Andrew Putman for explaining to us the proof of the BNS criterion given in \cite{Str}
and to the anonymous referee who made a number of suggestions that helped improve the exposition.
We also thank Thomas Church, Thomas Koberda and Andrew Putman for useful discussions related to the subject of this paper.
After this paper was completed, the authors learned that results similar to those in this paper were independently obtained by Church and Putman (unpublished).

\section{BNS invariant and effective finite generation}
\label{sec:EFG}
We begin this section with some basic terminology. Let $G$ be a group and $S$ a subset of $G$. 
\vskip .12cm

\paragraph{\bf Cayley graphs} The {\it Cayley graph} of $G$ with respect to $S$, denoted by $Cay(G,S)$, is the graph whose vertex set is $G$ and where $g,h\in G$ are connected by an edge if and only if $h=gs^{\pm 1}$ for some $s\in S$.
It is clear that $Cay(G,S)$ is connected if and only if $S$ generates $G$.

\vskip .12cm
\paragraph{\bf $S$-words} By an {\it $S$-word}, we will mean a formal expression $s_1\ldots s_k$ with $s_i\in S\cup S^{-1}$. Thus each $S$-word naturally represents an element of $G$, and every element of $G$ is represented by some $S$-word if and only if $S$ generates $G$.
For each $g\in G$, there is a natural bijection between $S$-words representing $g$ and paths in $Cay(G,S)$ from $1$ to $g$. 

\vskip .12cm
\paragraph{\bf Prefixes}
If $w=s_1\ldots s_k$ is an $S$-word representing $g\in G$, by a {\it prefix} of $w$ we will mean a subword of the form $s_1\ldots s_l$ with $l\leq k$. Thus, geometrically, prefixes of $w$ correspond to initial segments of the corresponding path in $Cay(G,S)$ from $1$ to $g$.

\vskip .12cm
\paragraph{\bf Word length} 
If $S$ generates $G$, for each $g\in G$ we denote by $\|g\|_S$ the word length of $g$ with respect to $S$, that is, the smallest
$k\in\dbZ_{\geq 0}$ such that $g$ is represented by an $S$-word $s_1\ldots s_k$. Geometrically, $\|g\|_S$ is the distance
from $1$ to $g$ in $Cay(G,S)$.

\subsection{Review of the BNS invariant}
\label{sec:BNSreview}

We start by recalling the definition of the BNS invariant.
By a character of a group $G$ we will mean a homomorphism from $G$ to the additive group of $\dbR$. Two characters $\chi$ and $\chi'$ will be considered equivalent if they are positive multiples of each other, and the equivalence class of a character $\chi$ will be denoted by $[\chi]$.
The character sphere $\dbS(G)$ is the set of equivalence classes of nonzero characters of $G$.

Assume now that $G$ is generated by a finite set $S$.
Given a character $\chi$ of $G$, denote by $Cay(G,S)_{\chi}$ the full subgraph of $Cay(G,S)$ with vertex set
$\{g\in G: \chi(g)\geq 0\}$. Note that $Cay(G,S)_{\chi}$ is completely determined by the equivalence class of $\chi$.
The BNS invariant of $G$, denoted by $\Sigma(G)$, is defined by
$$\Sigma(G)=\{[\chi]\in \dbS(G): Cay(G,S)_{\chi} \mbox{ is connected}\}.$$
It is not hard to show that $\Sigma(G)$ does not depend on the choice of $S$ although this is not obvious from definition.

The following remarkable result was proved by Bieri, Neumann and Strebel in \cite{BNS}:

\begin{Theorem}[BNS criterion]
\label{thm:BNS} Let $\K$ be a normal subgroup of $G$ such that $G/\K$ is abelian. Then $\K$ is finitely generated if and only if $\Sigma(G)$ contains $[\chi]$ for every character $\chi$ which vanishes on $\K$. In particular, $[G,G]$ is finitely generated if and only if $\Sigma(G)=\dbS(G)$.
\end{Theorem}

The original proof of Theorem~\ref{thm:BNS} given in \cite{BNS} was quite involved. A much simpler and more transparent proof appears in an unpublished manuscript of Strebel \cite{Str} who attributes the argument to Bieri. While still ineffective, the proof in \cite{Str} is almost entirely algorithmic apart from one step, as we will explain later in this section.

\subsection{On the proof of the BNS criterion}
\label{sec:BNSoutline}
In this subsection we will give a brief outline of the proof of the ``if'' part of Theorem~\ref{thm:BNS} from \cite{Str}.
With the exception of Lemma~\ref{lem:Nielsen} below, the results discussed in this subsection will not be used in the rest of the paper, and
the main purpose of providing this outline is to help the reader follow the proofs later in this section where we will establish an
effective version of (the ``if'' part of) Theorem~\ref{thm:BNS} under some additional hypotheses.
\vskip .1cm

The following theorem (Theorem~\ref{preimage}) must be well known, although we are not aware of a reference in the literature where it is stated exactly in this form. We are grateful to Andrew Putman for pointing out the formulation below.

\begin{Theorem} 
\label{preimage}
Let $G$ be a group generated by a finite set $S$,  let $\K$ be a subgroup of $G$ (not necessary normal), and let $\theta:G\to G/\K$ be the natural projection. Then $\K$ is finitely generated if and only if there is a finite subset $A$ of $G/\K$ such that
$\theta^{-1}(A)$ is connected in $Cay(G,S)$.
\end{Theorem}
The ``only if'' part (which is not essential for our purposes) is a straightforward exercise.
The ``if'' part of Theorem~\ref{preimage} is an immediate consequence of \cite[Theorem~A4.7]{Str}. 
Later in this section we will prove Theorem~\ref{lem:ReidSchreier} which is an effective version of 
the ``if'' part of Theorem~\ref{preimage}.

We now begin a sketch of proof of the ``if'' direction of Theorem~\ref{thm:BNS}.
Keeping all the notations from Theorem~\ref{preimage}, suppose now that $G/\K$ is abelian and $\Sigma(G)\supseteq \dbS(G/\K)$.
We wish to show that $\K$ is finitely generated.
Since $G$ is finitely generated, after replacing $K$ by a finite index overgroup (which does not affect finite generation),
we can assume that $G/\K$ is torsion-free. In addition, we want to impose an extra condition on the generating set $S$ given
by \eqref{eq:hyp1} below.

\begin{Lemma} 
\label{lem:Nielsen}
Let $G$ be a finitely generated group and let $\K$ be a normal subgroup of $G$ such that $G/\K$ is abelian and torsion-free.
Let $\theta:G\to G/K$ be the natural projection, and choose a basis $E$ of $G/K$. Then $G$ has a generating set $S$
such that
\begin{equation}
\label{eq:hyp1}
\theta(S)=E \quad \mbox {\rm or }\quad \theta(S)=E\cup\{0\}.
\end{equation}
Moreover, if $S_0=\{s_1,\ldots, s_n\}$ is any finite generating set
of $G$, one can obtain a generating set $S$ satisfying \eqref{eq:hyp1} from $S_0$ by applying a sequence of right Nielsen transformations, that is, transformations of the form 
$(g_1,\ldots,g_i,\ldots g_n)\mapsto (g_1,\ldots, g_ig_j^{\pm 1},\ldots, g_n)$ for some $i\neq j$ and
possibly invering one of the generators. 
\end{Lemma}
\begin{proof} Write $E=\{e_1,\ldots, e_m\}$.
Since the sequence $(\theta(s_1),\ldots, \theta(s_n))$ generates $G/K$, 
by basic linear algebra,\footnote{This can be done by writing $\theta(s_i)=\sum_{i=1}^m a_{ij}e_j$ with $a_{ij}\in\dbZ$
for each $1\leq i\leq n$ and turning the matrix $(a_{ij})$ into Smith Normal Form using a standard algorithm.}
using transformations of the form 
$(v_1,\ldots,v_i,\ldots v_n)\mapsto (v_1,\ldots, v_i\pm v_j,\ldots, v_n)$ for some $i\neq j$
we can reduce $(\theta(s_1),\ldots, \theta(s_n))$ to a sequence of the form
$(\pm e_1,e_2,\ldots, e_m,0,\ldots, 0)$.
Since each such linear transformations is induced by a right Nielsen transformation, applying the corresponding sequence of right Nielsen transformations to $S_0$ and then replacing one of the generators by its inverse if needed,
we obtain a generating set $S$ with desired properties.  
\end{proof}

From now on assume that $S$ satisfies the conclusion of Lemma~\ref{lem:Nielsen} (with respect to some fixed basis $E$).
Choose an isomorphism $G/\K\cong \dbZ^m$ which maps $E$ onto the standard basis of $\dbZ^m$. Let
$\|\cdot\|$ denote the corresponding $l^2$-norm on $G/\K$, and let $B(R)$ denote the $l^2$-ball of radius $R$ with respect to this norm (centered at $0$).
The goal now is to show that $\theta^{-1}(B(R))$ is connected for sufficiently large $R$ (this would imply that
$\K$ is finitely generated by Theorem~\ref{preimage}). To do this, one chooses an arbitrary path $p$ in $Cay(G,S)$ whose end vertices $a$ and $b$ lie in $\theta^{-1}(B(R))$ and then applies a (finite) sequence of modifications to $p$, so that the resulting path lies entirely in $\theta^{-1}(B(R))$.

\vskip .12cm

At each step the modification is as follows. Choose a vertex $g$ on the current path such that $\|\theta(g)\|$ is maximal.
If $\|\theta(g)\|\leq R$, there is nothing to do, so assume that $\|\theta(g)\|> R$. Let $gy_1$ and $gy_2$ be the vertices preceding and succeeding $g$ on this path. Define the character $\chi_g$ of $G$ by 
$\chi_g(x)=-(\theta(g),\theta(x)),$ 
and choose $t\in S^{\pm 1}$ such that $\chi_g(t)>0$. Since by assumption $\|\theta(gy_i)\|\leq \|\theta(g)\|$ for $i=1,2$
and 
$$\|\theta(gy_i)\|^2=\|\theta(g)+\theta(y_i)\|^2=\|\theta(g)\|^2+\|\theta(y_i)\|^2+2(\theta(g),\theta(y_i))
=\|\theta(g)\|^2+\|\theta(y_i)\|^2-2\chi_g(y_i),$$
we deduce that $\chi_g(y_i)\geq 0$ for $i=1,2$. 

Since $\chi_g$ vanishes on $\K$ and we assume that $\Sigma(G)\supseteq \dbS(G/\K)$, there exist paths $p_{t,y_2,g}$ from $t$ to $y_2 t$ and $q_{t,y_1,g}$ from $y_1 t$ to $t$ such that $\chi_g$ is positive on any vertex of those paths. Now we replace the segment $(gy_1,g,gy_2)$ of the current path by a new subpath passing through $gy_1, gy_1t, gt, gy_2 t, gy_2$ where one moves from $gy_1 t$ to $gt$ using the path $g\circ q_{t,y_1,g}$ and from $gt$ to $gy_2 t$ using the path $g  \circ p_{t,y_2,g}$
(see Figure~\ref{figure:pathoutline}).

\begin{figure}
\begin{tikzpicture}
  [fil/.style={circle, fill=black, inner sep=0pt, outer sep=0pt, minimum size=1.5mm, thick}, op/.style={circle, draw, fill=white, inner sep=0pt, outer sep=0pt, minimum size=1.5mm}]

  \node[fil, label=120:$gy_1$] (gy1) at (0,0){};
  \node[fil, label=above:$g$] (g) at (4,1){};
  \node[fil, label=60:$gy_2$] (gy2) at (8,0){};
  
  \draw[dashed] (-1,-0.5) -- (gy1) (gy2) -- (9,-0.5);
  \draw[thick] (gy1) -- (g) node[pos=0.5, above]{{\footnotesize $y_1$}} -- (gy2) node[pos=0.5, above]{{\footnotesize $y_2$}};

  \node[op, label=below:$gy_1t$] (gy1t) at (1, -3){};
  
  \draw[red] (gy1) -- (gy1t) node[black, pos=0.4, right]{{\footnotesize $t$}};

  \node[op, label=below:$gy_2t$] (gy2t) at (7, -3){};
  
  \draw[red] (gy2) -- (gy2t) node[black, pos=0.4, left]{$t$};

  \node[op, label=above:$gt$] (gt) at (4, -2){};
  
  \node[label=below:{\small $g\circ q_{t, y_1, g}$}] at (2.7,-2.3){};
  \node[label=below:{\small $g\circ p_{t, y_2, g}$}] at (5.4,-2.3){};

  \draw[red, dashed] (gy1t.30) sin (gt.west) (gt.east) cos (gy2t.150);

\end{tikzpicture}
\caption{}
\label{figure:pathoutline}
\end{figure}

A key step of the proof is a compactness argument\footnote{In order to use compactness one has to replace $\chi_g$ by the normalized character $\chi'_g=\frac{1}{\|\theta(g)\|}\chi_g$
(note that $\chi_g$ and $\chi'_g$ are equivalent). If we extend the chosen inner product from $G/K$ to $G/K\otimes \dbR$
and consider characters of $G$ as elements of the dual space $\Hom(G/K\otimes \dbR,\dbR)$, then characters of the form
$\chi'_g$ always have norm $1$.}
which shows that there is a finite set of paths $\Omega$ and $\eps>0$
such that for any character $\chi_g$ arising above we can find desired paths $p_{t,y_2,g}$ and $q_{t,y_1,g}$ in $\Omega$ and moreover
$\chi_g(v)\geq \eps\|\theta(g)\|$ for any vertex $v$ of $p_{t,y_2,g}$ or $q_{t,y_1,g}$. Now let $r$ be the maximum of $\|\theta(v)\|$
where $v$ ranges over the vertices of all paths from $\Omega$. The direct computation below shows  
that if $R\geq\frac{r^2}{2\eps}$, then $\|\theta(z)\|<\|\theta(g)\|$ for any vertex $z$ on the newly added segment. This concludes the (sketch of) proof of the ``if'' direction of Theorem~\ref{thm:BNS}.

Indeed, any new vertex $z$ has the form $gv$ where $v$ lies on $q_{t,y_1,g}$ or $p_{t,y_2,g}$. Therefore,
\begin{align*}
\|\theta(z)\|^2=\|\theta(g)+\theta(v)\|^2 & =\|\theta(g)\|^2+\|\theta(v)\|^2+2(\theta(g),\theta(v)) \\
& = \|\theta(g)\|^2+\|\theta(v)\|^2-2\chi_g(v)\\
& \leq \|\theta(g)\|^2+r^2-2\eps\|\theta(g)\| \\
& \leq \|\theta(g)\|^2+2\eps(R-\|\theta(g)\|)<\|\theta(g)\|^2.   
\end{align*}
\skv
As the above outline suggests, in order to turn this proof into an actual algorithm (for a specific group), one needs to have an explicit procedure for constructing the paths $p_{t,y_2,g}$ and $q_{t,y_1,g}$. We will now discuss some additional conditions which make this possible.

\subsection{Explicitly finding non-negative forms}

The following notion of a non-negative form of a group element provides a convenient way to reformulate the definition
of $\Sigma(G)$.

\begin{Definition}\rm Let $S$ be a generating set of a group $G$, let $\chi$ be a nonzero character of $G$ and assume that $\chi(g)\geq 0$ for some $g\in G$. By a {\it $(\chi,S)$-non-negative form of $g$} we will mean an $S$-word $w$ which represents $g$ such that $\chi(v)\geq 0$ for every prefix $v$ of $w$.
\end{Definition}

It is clear that a given $g\in G$ with $\chi(g)\geq 0$ admits a $(\chi,S)$-non-negative form if and only if there is a path in $Cay(G,S)_{\chi}$ connecting $1$ with $g$ (recall that $Cay(G,S)_{\chi}$ is the full subgraph of $Cay(G,S)$ with the vertex set
$\{x: \chi(x)\geq 0\}$). Thus, $[\chi]\in \Sigma(G)$ if and only if every $g$ with $\chi(g)\geq 0$ admits a $(\chi,S)$-non-negative form.
\medskip

There is a well-known sufficient condition for a character to lie in the BNS invariant:

\begin{Lemma}
\label{lem:KMM}
Suppose $G$ is generated by $S=\{s_1,\ldots, s_n\}$ and $\chi$ is a character of $G$ such that
\begin{itemize}
\item[(i)] $\chi(s_1)>0$;
\item[(ii)] for every $i\geq 2$ there exists $j<i$ such that $[s_j,s_i]=1$ and $\chi(s_j)\neq 0$.
\end{itemize}
Then $[\chi]\in \Sigma(G)$.
\end{Lemma}

Lemma~\ref{lem:KMM} was proved in \cite[Lemma~1.9]{KMM}, although indirectly it appeared already in \cite{MV}; see also \cite[Lemma~2.4]{EH} for a generalization. Although the condition in Lemma~\ref{lem:KMM} may appear very special, there are many important classes of groups $G$ for which every character in $\Sigma(G)$ does satisfy this condition for some $S$ -- for instance, this is the case for right-angled Artin groups \cite{MV}, groups of pure symmetric automorphisms of free groups \cite{OK}, pure braid groups 
\cite{KMM}, $\IA_n$ for $n\geq 5$ and $\calI_n^b$ for $n\geq 5$, $b=0,1$ (see \cite{CEP, EH}).

If a generating set $S$ and a character $\chi$ satisfy the hypotheses of Lemma~\ref{lem:KMM}, it is not difficult to describe an algorithm which computes an $(S,\chi)$-non-negative form for a given $g\in G$ with $\chi(g)\geq 0$. 
In fact, such an algorithm implicitly appears in the proof of \cite[Theorem~4.1]{MV}. 

In this paper it will be more convenient to work with a slightly more restrictive condition, which still holds for
$\IA_n$ and $\calI_n^b$ for sufficiently large $n$ and leads to a very simple formula for a $(\chi,S)$-non-negative form.

We will need some technical definitions.

\begin{Definition}\rm Let $S$ be a finite generating set for a group $G$ and $Z$ a subset of $S$. 
\begin{itemize}
\item[(a)] We will say that the pair $(S,Z)$ is \emph{chain-centralizing} if for every $s\in S$ and $z\in Z$ there exists $z'\in Z$ which commutes with both $s$ and $z$. 
\item[(b)] If $\chi$ is a character of $G$, we will say that $\chi$ is {\it regular} for $(S,Z)$ if $\chi(z)\neq 0$ for all $z\in Z$.
\end{itemize}
\end{Definition}
\begin{Remark}\rm If $\chi$ is regular for a chain-centralizing pair $(S,Z)$, it is easy to show that the pair $(\chi,S)$
satisfies the hypothesis of Lemma~\ref{lem:KMM} (for a suitable ordering of $S$), but we will not use this fact in the proofs.
\end{Remark}

The name \emph{chain-centralizing} is motivated by the following property which is an obvious consequence of the definition:
\begin{Observation}
\label{obs:cc}
Suppose that $(S,Z)$ is chain-centralizing. Then for any finite sequence $s_1,\ldots, s_k\in S$ there exists a sequence $z_1,\ldots, z_k\in Z$ such that
\begin{itemize}
\item[(i)] $z_i$ commutes with $s_i$ for each $1\leq i\leq k$;
\item[(ii)]$z_i$ commutes with $z_{i-1}$ for each $2\leq i\leq k$.
\end{itemize}
\end{Observation}

The following lemma shows that if $\chi$ is regular for a chain-centralizing pair $(S,Z)$, it is very easy to construct
$(\chi,S)$-non-negative forms:

\begin{Lemma}
\label{lemma:cc}
Suppose that $(S,Z)$ is chain-centralizing and $\chi$ is regular for $(S,Z)$. 
Let $g\in G$ with $\chi(g)\geq 0$, and write $g=s_1\ldots s_k$ with $s_i\in S^{\pm 1}$. Choose $z_1,\ldots, z_k$ satisfying conditions (i) and (ii) of Observation~\ref{obs:cc}, and choose $n_i\in\dbZ$ such that 
\begin{itemize}
\item[(a)] $n_i \chi(z_i)\geq 0$ for all $1\leq i\leq k$;
\item[(b)] $n_i \chi(z_i)\geq -\sum\limits_{j=1}^i \chi(s_j)$ for all $1\leq i\leq k$;
\item[(c)] $n_i \chi(z_i)\geq -\sum\limits_{j=1}^{i-1} \chi(s_j)$ for all $2\leq i\leq k$
\end{itemize} 
(such $n_i$ exist since $\chi$ is regular). Then the $S$-word
$$w_{\chi}=z_1^{n_1}s_1 z_2^{n_2}z_1^{-n_1}s_2 z_3^{n_3}z_2^{-n_2}s_3\ldots z_k^{n_k}z_{k-1}^{-n_{k-1}}s_k z_k^{-n_k}$$
is a $(\chi,S)$-non-negative form of $g$.
\end{Lemma}
\begin{proof} The word $w_{\chi}$ represents $g$ by conditions (i) and (ii) of Observation~\ref{obs:cc} -- we first move $z_1^{-n_1}$ past $s_1$
and $z_2^{n_2}$ and cancel it with $z_1^{n_1}$, then we move $z_2^{-n_2}$ past $s_2$
and $z_3^{n_3}$  and cancel it with $z_2^{n_2}$ etc. 

Let us now prove that $\chi(v)\geq 0$ for every $S$-prefix $v$ of $w_{\chi}$. Without loss of generality, we can assume that
$\chi(z_i)> 0$ for all $i$, in which case $n_i\geq 0$ by (a). If $v$ does not end with $z_i^{-n_i}$ or $s_i$ for some $i$, 
we can produce another $S$-prefix $v'$ of $w_{\chi}$ with $\chi(v')\leq \chi(v)$ by either removing the last letter of $v$ or adding the next letter of $w_{\chi}$ to the end of $v$. Thus, it suffices to prove that $\chi(v)\geq 0$ when $v$ ends with $z_i^{-n_i}$ or $s_i$.

\vskip .1cm
{\it Case 1:} $v=z_1^{n_1}\prod_{j=1}^{i}(s_{j}z_{j+1}^{n_{j+1}}z_j^{-n_j})$ for some $1\leq i\leq k-1$. In this case
$\chi(v)=n_1\chi(z_1)+\sum_{j=1}^{i}(\chi(s_j)+n_{j+1}\chi(z_{j+1})-n_j\chi(z_j))=
n_{i+1}\chi(z_{i+1})+\sum_{j=1}^{i} \chi(s_j)\geq 0$ by (c).

\vskip .1cm
{\it Case 2:} $v=z_1^{n_1}\prod_{j=1}^{i}(s_{j}z_{j+1}^{n_{j+1}}z_j^{-n_j})s_{i+1}$ for some $0\leq i\leq k-1$ (the case
$i=0$ corresponds to the empty product). Then $\chi(v)=n_{i+1}\chi(z_{i+1})+\sum_{j=1}^{i+1} \chi(s_j)\geq 0$ by (b).

\vskip .1cm
{\it Case 3:} $v$ is the full word $w_{\chi}$. In this case $\chi(v)=\chi(g)\geq 0$. 
\end{proof}

\subsection{Extra hypothesis}
\label{sec:extrahyp}

In this subsection we introduce the additional condition that will allow us to turn the proof of the BNS criterion into an algorithm. As before, we will assume that $G,K,E,\theta$ and $S$ satisfy the hypotheses and conclusion of Lemma~\ref{lem:Nielsen}:
$G/K$ is abelian and torsion-free, $\theta:G\to G/K$ is the natural projection,
$E$ is a fixed basis of $G/K$ and $S$ is a generating set of $G$ such that $\theta(S)=E$ or $E\cup\{0\}$.

In order to make use of Lemma~\ref{lemma:cc}, we need to know that every nonzero character is regular with respect to some chain-centralizing pair. Note that a single chain-centralizing pair would rarely work for all the characters (apart from rather trivial examples). Also observe that if we have one chain-centralizing pair $(S,Z)$, then for any $\phi\in \Aut(G)$,
the pair $(\phi(S),\phi(Z))$ is also chain-centralizing. 

This motivates our new hypothesis. We would like to assume that there is a finite subset $\Phi\subseteq \Aut(G)$ with the following property: for every nonzero character $\chi$ of $G$ which vanishes on $\K$, there is some $Z\subseteq S$ and
$\phi\in\Phi$ such that $(S,Z)$ is chain-centralizing and $\chi$ is regular for $(\phi(S),\phi(Z))$. In fact, we will need to assume a bit more (see the Regularity Hypothesis below), but first we will introduce some additional notations involving automorphisms of $G$.
\vskip .12cm

\paragraph{\bf Constants $A$ and $B$} Let $\phi\in \Aut(G)$. Define 
$$B(\phi,S)=\max\{\|\theta(\phi(s))\|_1 : s\in S\}$$
(where $\|\cdot\|_1$ denotes the $l^1$-norm with respect to $E$). 

\vskip .12cm

Now define $A=A(\phi,S)$ to be the smallest integer with the following property: for every $s\in S$, there is an $S$-word $w_{\phi,s}$ representing $\phi(s)$ such that $\|\theta(v)\|_1\leq A$ for every $S$-prefix $v$ of $w_{\phi,s}$. 
Since $w_{\phi,s}$ is its own $S$-prefix, we have the obvious inequality $B(\phi, S)\leq A(\phi, S)$.
\vskip .12cm

If $\Phi$ is a finite subset of $\Aut(G)$, we define 
$$A(\Phi,S)=\max\{A(\phi,S):\phi\in\Phi\}\quad\mbox{ and }\quad B(\Phi,S)=\max\{B(\phi,S):\phi\in\Phi\}.$$ 

When $S$ is fixed or clear from the context, we will usually suppress it from the notation and write $A(\phi)$ for $A(\phi,S)$ etc.

\vskip .12cm

\paragraph {\bf Constant $M$} For each character $\chi$ define 
$$M(\chi)=M(\chi,S)=\max\{|\chi(s)|: s\in S\}.$$ 
We are now ready to state our additional hypothesis. Let $\Aut(G,K)$ denote the subgroup of $\Aut(G)$ consisting of automorphisms
which leave $K$ invariant.

\begin{Hypothesis} 
\label{hyp2}
There exist  a finite subset $\Phi\subseteq \Aut(G,K)$  and a constant $C>0$ with the following property:
\begin{itemize}
\item[(***)] For every nonzero character $\chi\in Hom(G/\K,\dbR)$ there exist a subset $Z\subseteq S$ and
$\phi\in\Phi$ such that $(S,Z)$ is chain-centralizing and 
$|\chi(\phi(z))|\geq \frac{M(\chi)}{C}$ for all $z\in Z$.
\end{itemize}
\end{Hypothesis}
\begin{Remark}\rm The inequality in (***) implies in particular that $\chi$ is regular for $(\phi(S),\phi(Z))$.
\end{Remark}

Before proceeding, we establish a few simple inequalities involving the constants $A,B$ and $M$.

\begin{Observation} 
\label{obs:theta}
Let $g\in G$. Then $\|\theta(g)\|_1$ is the smallest integer $m$ for which there exists $h\in G$ with $\|h\|_S=m$ and $\theta(h)=\theta(g)$.
\end{Observation}
\begin{proof}
This immediately follows from the assumption that $\theta(S)=E$ or $E\cup\{0\}$.
\end{proof}

\begin{Claim} 
\label{cor:M}
Let $\chi$ be a character of $G$ which vanishes on $K$. Then $$|\chi(g)|\leq \|\theta(g)\|_1\cdot M(\chi)$$ for all $g\in G$.
\end{Claim}
\begin{proof}
Let $m=\|\theta(g)\|_1$. By Observation~\ref{obs:theta} there exist $s_1,\ldots, s_m\in S^{\pm 1}$
such that $\theta(g)=\theta(s_1\ldots s_m)$. Since the character $\chi$ vanishes on $K$, it factors through $\theta$
and thus 
$$|\chi(g)|=|\chi(s_1\ldots s_m)|\leq \sum_{i=1}^m |\chi(s_i)|\leq m\cdot \max\{\chi(s_i)\}\leq mM(\chi).
\quad\qedhere$$ 
\end{proof}

\begin{Claim} 
\label{claim1}
Let $\phi\in \Aut(G,K)$, and let $A=A(\phi^{-1})$ and $B=B(\phi)$.
Then for every $s\in S^{\pm 1}$ there there exists a $\phi(S)$-word $\widetilde s$ representing $s$
such that $\|\theta(w)\|_1\leq AB$ for every $\phi(S)$-prefix $w$ of $\widetilde s$. 
\end{Claim}
\begin{proof} 

Fix $s\in S^{\pm 1}$. By definition of the constant $A$ there exists an $S$-word
$x_1\ldots x_k$ (with $x_i\in S^{\pm 1}$) representing $\phi^{-1}(s)$ such that 
$\big\|\theta(\prod_{j=1}^l x_j)\big\|_1\leq A$ for all $1\leq l\leq k$.

Let $\widetilde s=\phi(x_1)\ldots\phi(x_k)$. Then $\widetilde s$ is a $\phi(S)$-word representing $\phi(\phi^{-1}(s))=s$.
Thus, to prove the claim it suffices to show that
\begin{equation}
\label{eq:S2}
\big\|\theta(\prod_{j=1}^l \phi(x_j))\big\|_1=\big\|\theta(\phi(\prod_{j=1}^l x_j))\big\|_1\leq AB\mbox{ for all }1\leq l\leq k.
\end{equation}
Fix such an index $l$. By Observation~\ref{obs:theta} and the choice of the word $x_1\ldots x_k$, there exists an $S$-word $v=y_1\ldots y_m$ with $m\leq A$ such that 
$\theta(\prod_{j=1}^l x_j)=\theta(v)$ and hence $\prod_{j=1}^l x_j =v z$ with $z\in K$.

Then $\phi(\prod_{j=1}^l x_j) =\phi(v) \phi(z)$, and since $K$ is $\phi$-invariant, we have 
$\theta(\phi(\prod_{j=1}^l x_j)) =\theta(\phi(v))$. Note that
$$\big\|\theta(\phi(v))\big\|_1=\big\|\theta(\phi(y_1\ldots y_m))\big\|_{1}\leq
\sum_{i=1}^m \big\|\theta(\phi(y_i))\big\|_1.$$
Since $m\leq A$ and $\big\|\theta(\phi(y_i))\big\|_1\leq B$ by definition of $B$, we get that
$\big\|\theta(\phi(v))\big\|_1\leq AB$ which proves \eqref{eq:S2}.
\end{proof}

\subsection{The main result}
In this subsection we will prove the main result of this section, Theorem~\ref{prop:connectedradius}. We start with a key proposition which shows that under the Regularity Hypothesis, we can control the ``size'' of 
$(\chi,S)$-non-negative forms with respect to any character $\chi$ of $G/\K$.

\begin{Proposition}\label{prop:path}
Let $G,\K$ and $S$ be as in Lemma~\ref{lem:Nielsen}, and assume that the Regularity Hypothesis holds 
for some $\Phi\subseteq \Aut(G,K)$  and a constant $C>0$.
Let $A=A(\Phi\cup\Phi^{-1})$ and $B=B(\Phi\cup\Phi^{-1})$. Let $\chi$ be a nonzero character of $G/\K$, let $g\in G$ with $\chi(g)\geq 0$, and write $g=s_1\ldots s_r$ with $s_i\in S^{\pm 1}$. 
Then there exists a $(\chi,S)$-non-negative form $w$ of $g$ such that for any prefix $v$ of $w$ we have
$$\|\theta(v)\|_1\leq (2BC+1)(AB+A+r)+2A.$$
 \end{Proposition}
\begin{proof} The basic idea is very simple.
Choose $\phi\in\Phi$ such that (***) in the Regularity Hypothesis holds for $\chi$, express $g$ as a $\phi(S)$-word and then use 
Lemma~\ref{lemma:cc} to construct a $(\chi,\phi(S))$-non-negative form of $g$.
This almost works -- the issue is that when we rewrite the obtained $(\chi,\phi(S))$-non-negative form as an $S$-word, prefixes of this $S$-word may have negative $\chi$-values. However, we have a lower bound for those $\chi$-values: $\chi(v)\geq -MA$ for every such prefix $v$, where $M=M(\chi)$.

To resolve this problem we choose $t\in S^{\pm 1}$ with $\chi(t)=M$ and apply the same argument to 
the element $t^{-A}gt^{A}=
t^{-A}\prod\limits_{i=1}^r s_i\,\, t^A$. It is easy to see that if $w$ is an $S$-word representing $t^{-A}gt^{A}$ such that
$\chi(v)\geq -MA$ for every such prefix $v$ of $w$, then $t^A w t^{-A}$ is a $(\chi,S)$-non-negative from of $g$.

We now present the full argument. For convenience, we break the construction into three steps.
\vskip .12cm

{\it Step 1:} Choose $t\in S^{\pm 1}$ with $\chi(t)=M$. Recall that $s_1\ldots s_r$ is an $S$-word representing
$g$ and hence $p=t^{-A}s_1\ldots s_r t^A$ is an $S$-word representing $t^{-A}gt^{A}$. Replacing
each factor $s\in \{s_1,\ldots,s_r,t^{\pm 1}\}$ in the word $p$ by the corresponding $\widetilde s$ 
given by Claim~\ref{claim1}, we obtain a $\phi(S)$-word $\widetilde p$
representing $t^{-A}gt^{A}$. 
\vskip .12cm

Let $v$ be an arbitrary $\phi(S)$-prefix of $\widetilde p$ . We claim that 
\begin{itemize}
\item[(1)] $\|\theta(v)\|_1\leq AB+A+r$;
\item[(2)] $|\chi(v)|\leq (AB+A+r)M$.
\end{itemize}
For convenience let us write $t^{-A}gt^{A}=\prod\limits_{i=1-A}^{r+A}s_i$ where $s_i=t^{-1}$ for $i\leq 0$ and $s_i=t$ for $i>r$.
Thus, $\widetilde p=\prod\limits_{i=1-A}^{r+A}\widetilde s_i$, and
there exists some $-A\leq j< r+A$ and a $\phi(S)$-prefix $w$ of $\widetilde s_{j+1}$ such that $v=(\prod_{i\leq j}\widetilde s_i) w$.
Then 
$$\|\theta(v)\|_1 \leq \|\theta(\prod_{i\leq j}\widetilde s_i)\|_1+ \|\theta(w)\|_1=
\|\sum_{i\leq j}\theta(s_i)\|_1+ \|\theta(w)\|_1.$$ 
The first summand is bounded above by
$A+r$ (this upper bound may occur for $j=r$ since for larger $j$ we start getting cancellations of $\theta(t)$ with 
$\theta(t^{-1})$). Also $\|\theta(w)\|_1\leq AB$ by Claim~\ref{claim1}, so we proved (1).
Inequality (2) immediately follows from (1) and Claim~\ref{cor:M}.
\vskip .2cm

{\it Step 2:} Next we use Lemma~\ref{lemma:cc} to construct a $(\chi,\phi(S))$
non-negative form of $\widetilde p$, call it $\widetilde q$ (the elements $z_i$ in Lemma~\ref{lemma:cc} will be chosen from $\phi(Z)$). For each $i$
we shall choose $n_i$ to be smallest in absolute value satisfying the inequality in Lemma~\ref{lemma:cc}. Since $|\chi(z_i)|\geq \frac{M}{C}$ by the Regularity Hypothesis and $|\chi(v)|\leq (AB+A+r)M$ for every prefix $v$ of $\widetilde p$, we have
$|n_i|\leq C(AB+A+r)$.

Now we need to establish the analogue of (1) for prefixes of $\widetilde q$. Let $v$ be a $\phi(S)$-prefix of $\widetilde q$. It is clear from the construction that there is a $\phi(S)$-prefix $w$ of $\widetilde p$ such that $\theta(v)$ is equal to $\theta(w)$ or $\theta(wz_i^{m})$ with
$|m|\leq |n_i|$ for some $i$ or $\theta(wz_i^{n_i}z_{i+1}^m)$ with $|m|\leq |n_{i+1}|$ for some $i$. In any case
we have $$\|\theta(v)\|_1 \leq  \|\theta(w)\|_1+ |n_i|\cdot \|\theta(z_i)\|_1 +|n_{i+1}|\cdot \|\theta(z_{i+1})\|_1.$$
Since $\|\theta(w)\|_1\leq AB+A+r$ by (1), $\|\theta(z_j)\|_1\leq B$ for all $j$ by definition of $B$ and
$|n_j|\leq C(AB+A+r)$ for all $j$ as established above, we get
\begin{itemize}
\item[(3)]  $\|\theta(v)\|_1\leq (AB+A+r)+2BC(AB+A+r)=(2BC+1)(AB+A+r)$.
\end{itemize}
\vskip .2cm

{\it Step 3:} Next we rewrite $\widetilde q$ as an $S$-word, call it $q$. Let $v$ be an $S$-prefix of $q$. 
We claim that
\begin{itemize}
\item[(4)]  $\|\theta(v)\|_1\leq (2BC+1)(AB+A+r)+A$; 
\item[(5)] $\chi(v)\geq -MA$. 
\end{itemize}
Indeed, similarly to Step~1,
we can write $v=v_1 v_2$ where $v_1$ is a $\phi(S)$-prefix of $\widetilde q$ and $v_2$ is an $S$-prefix of some $x\in\phi(S^{\pm 1})$.
Now $\|\theta(v_1)\|_1\leq (2BC+1)(AB+A+r)$ by (3) and $\|\theta(v_2)\|_1\leq A$ by definition of $A$, so (4) holds.

Since $\|\theta(v_2)\|_1\leq A$, Claim~\ref{cor:M} yields $|\chi(v_2)|\leq MA$, so in particular $\chi(v_2)\geq -MA$. 
Since $\widetilde q$ is a $(\chi,\phi(S))$-non-negative form, we must have $\chi(v_1)\geq 0$, so 
$\chi(v)=\chi(v_1)+\chi(v_2)\geq -MA$. Thus we proved (5).

Recall now that $q$ represents $t^{-A}gt^A$ in $G$ and hence $t^A q t^{-A}$ represents $g$. We claim that $t^A q t^{-A}$ is the desired $(\chi,S)$-non-negative form. Take any $S$-prefix $w$ of $t^Aqt^{-A}$. We need to show that
\begin{itemize}
\item[(6)] $\|\theta(w)\|_1\leq (2BC+1)(AB+A+r)+2A$; 
\item[(7)]  $\chi(w)\geq 0$.
\end{itemize}
Clearly, there are 3 cases:

{\it Case 1:} $w=t^m$ for some $0\leq m\leq A$. In this case both (6) and (7) are obvious.

{\it Case 2:} $w=t^Av$ for some prefix $v$ of $q$. In this case $\|\theta(v)\|_1\leq (2BC+1)(AB+A+r)+A$ by (4), so (6) holds,  
and (7) follows from (5) and the fact that $\chi(t^A)=MA$ (by the choice of $t$).

{\it Case 3:} $w=t^A q t^{-m}$ for some $0\leq m\leq A$. In this case (6) again follows  from (4). Finally,
$\chi(t^A q t^{-A})\geq 0$ since $t^A q t^{-A}$ represents $g$ and hence
$\chi(w)=\chi(t^A q t^{-A}t^{A-m})\geq (A-m)\chi(t)\geq 0$.
\end{proof} 

We can now construct an explicit finite subset of $G/\K$ whose preimage in $Cay(G,S)$ is connected. We will show that the $l^{\infty}$-ball of a certain radius has this property:

\begin{Theorem}
\label{prop:connectedradius}
We keep all the hypotheses and notations from Proposition~\ref{prop:path} and
let $R=16(BC+1)^2(AB+3A+3)$. Let $B_{\infty}(R)$ be the $l^{\infty}$-ball of radius $R$ in $G/\K$. Then $\theta^{-1}(B_{\infty}(R))$ is connected in $Cay(G,S)$.
\end{Theorem}

To simplify terminology, in the proof of Theorem~\ref{prop:connectedradius} we will occasionally talk about 
the $l^2$-norm or $l^{\infty}$-norm for an element $g\in G$, by which we mean the corresponding norm of $\theta(g)$.

\begin{proof}
Recall that $E$ is a fixed a basis of $G/K$ such that $\theta(S)=E$ or $E\cup\{0\}$. 
As before, we choose an isomorphism $G/K\to \dbZ^m$ which maps $E$ to the standard basis of $\dbZ^m$.

We will generally follow the outline of the proof of Theorem~\ref{thm:BNS} given earlier in this section. Here is the summary of the new ingredients that we will use:
\begin{itemize}
\item[(i)] The compactness argument will be replaced by a reference to Proposition~\ref{prop:path}, which enables us to get
an explicit formula for $R$.
\item[(ii)] In the proof we will use not just $l^2$-norm, but also $l^{\infty}$-norm and $l^1$-norm (where all the norms
are taken with respect to the chosen identification of $G/K$ with $\dbZ^m$).
\item[(iii)] When modifying the path at each step we will make slightly more complicated ``detours''.
\end{itemize}
Note that modification (i) is essential, while (ii) and (iii) will only be used to obtain a better estimate.

So take any $a,b\in G$ with $\|\theta(a)\|_{\infty},\|\theta(b)\|_{\infty}\leq R$. Our goal is to show that there is a path
from $a$ to $b$ in $Cay(G,S)$ whose projection lies in $B_{\infty}(R)$. We start by choosing some path $p$ from $a$ to $b$
in $Cay(G,S)$ and then apply a sequence of modifications, eventually pushing it inside $\theta^{-1}(B_{\infty}(R))$.

If $\theta(p)$ lies inside $B_{\infty}(R)$, we are done, so assume that $\theta(p)$ has at least one vertex outside $B_{\infty}(R)$. Among all the vertices of $p$ outside $B_{\infty}(R)$ we choose one with the largest $l^2$-norm, call it $g$. Our goal is to replace $p$ by another path $p'$ from $a$ to $b$ which
does not have any vertices with $l^2$-norm larger than $\|g\|_2$ and has fewer vertices with $l^2$-norm equal to $\|g\|_2$ than $p$. Clearly, after applying such a modification finitely many times, we will obtain a path inside $\theta^{-1}(B_{\infty}(R))$, as desired. Note that the maximal $l^{\infty}$-norm may increase during some initial steps -- this is not a problem.

Define the character $\chi$ of $G/\K$ by $\chi(x)=-(\theta(g),\theta(x))$.  Let $gy_1$ and $gy_2$, where $y_1,y_2\in S\cup S^{-1}$, be the vertices preceding and succeeding $g$ in the path $p$, respectively. 
We claim that $||\theta(gy_i)||_2\leq ||\theta(g)||_2$ for $i=1,2$. 

To prove the claim recall that by assumption $g$ has the largest $l^2$-norm among the vertices of $p$ which lie 
outside of $\theta^{-1}(B_{\infty}(R))$. Thus, if 
$\theta(gy_i)\not \in B_{\infty}(R)$, it is automatic that $||\theta(gy_i)||_2\leq ||\theta(g)||_2$.
On the other hand, if $\theta(gy_i)\in B_{\infty}(R)$, we have 
$||\theta(gy_i)||_{\infty}\leq R< ||\theta(g)||_{\infty}$. 
But by assumption $\theta(y_i)\in\{\pm e_j\}\cup\{0\}$. Hence $\theta(g)$ and $\theta(gy_i)$ differ in at most one coordinate, so
$||\theta(gy_i)||_{\infty}< ||\theta(g)||_{\infty}$ forces $||\theta(gy_i)||_{2}< ||\theta(g)||_{2}$
as well.

\vskip .12cm

Thus, we proved that $||\theta(gy_i)||_2\leq ||\theta(g)||_2$ for $i=1,2$. Since   
$$||\theta(gy_i)||_2^2 =||\theta(g)||_2^2+||\theta(y_i)||_2^2+2(\theta(g), \theta(y_i))
	 = \|\theta(g)\|_2^2+||\theta(y_i)||_2^2 - 2\chi(y_i),$$
it follows that $\chi(y_i)\geq 0$ for $i=1,2$.

Let $\theta(g)_i$ denote the  $i^{\rm th}$ coordinate of $\theta(g)$, and choose $i$ such that 
$\theta(g)_i$ is maximal in absolute value. Since by assumption $\theta(g)\not\in B_{\infty}(R)$, we have
$|\theta(g)_i|> R$. Choose $t\in S^{\pm 1}$
with $\theta(t)=\pm e_i$ where the sign is chosen to be the same as the sign of $\theta(g)_i$. 
Then $(\theta(g),\theta(t))=|\theta(g)_i|> R$.

We proceed with the construction of $p'$ (see Figure~\ref{figure:pathproof} for an illustration). The path $p'$ will coincide with $p$ prior to $gy_1$ and after $gy_2$, but we will connect $gy_1$ and $gy_2$ by a new subpath which passes through the vertices $gy_1 t^{-r}, gt^{-r}$ and $gy_2 t^{-r}$, in this order, where $r$ satisfying $1\leq r\leq R$ will be chosen later. We connect $gy_1$ with $gy_1 t^{-r}$ in the natural way (multiplying $r$ times by $t^{-1}$) and similarly $gy_2 t^{-r}$ with
$gy_2$. To connect $gt^{-r}$ with $gy_2 t^{-r}$, we write $gy_2 t^{-r}=(gt^{-r})(t^r y_2 t^{-r})$ and then replace the suffix
$t^r y_2 t^{-r}$ by its $(\chi,S)$-nonnegative form constructed in Proposition~\ref{prop:path} (the latter is applicable since
$\chi(t^r y_2 t^{-r})=\chi(y_2)\geq 0$). Similarly we connect $gy_1 t^{-r}$ with $g t^{-r}$. 

\begin{figure}
% the p' segment highlighted in red

\begin{tikzpicture}
  [fil/.style={circle, fill=black, inner sep=0pt, outer sep=0pt, minimum size=1.5mm, thick}, op/.style={circle, draw, fill=white, inner sep=0pt, outer sep=0pt, minimum size=1.5mm}]

  \node[fil, label=120:$gy_1$] (gy1) at (0,0){};
  \node[fil, label=above:$g$] (g) at (4,1){};
  \node[fil, label=60:$gy_2$] (gy2) at (8,0){};
  
  \draw[dashed] (-1,-0.5) -- (gy1) (gy2) -- (9,-0.5);
  \draw[thick] (gy1) -- (g) node[pos=0.5, above]{{\footnotesize $y_1$}} -- (gy2) node[pos=0.5, above]{{\footnotesize $y_2$}};
  
  \node[op] (Lt1) at (0.33, -1){};
  \node[op] (Lt2) at (0.67, -2){};
  \node[op, label=below:$gy_1t^{-r}$] (gy1t) at (1, -3){};
  
  \draw[red] (gy1) -- (Lt1) node[black, pos=0.4, right]{{\footnotesize $t$}} (Lt2) -- (gy1t) node[black, pos=0.4, right]{{\footnotesize $t$}};
  \draw[red, dashed] (Lt1) -- (Lt2);
  
  \node[op] (Rt1) at (7.67, -1){};
  \node[op] (Rt2) at (7.33, -2){};
  \node[op, label=below:$gy_2t^{-r}$] (gy2t) at (7, -3){};
  
  \draw[red] (gy2) -- (Rt1) node[black, pos=0.4, left]{$t$} (Rt2) -- (gy2t) node[black, pos=0.4, left]{$t$};
  \draw[red, dashed] (Rt1) -- (Rt2);
  
  \node[op, label=above:$gt^{-r}$] (gt) at (4, -2){};
  
  \draw[red, dashed] (gy1t.30) sin (gt.west) (gt.east) cos (gy2t.150);

\end{tikzpicture}
\caption{}
\label{figure:pathproof}
\end{figure}

We now need to show that if $v$ is any vertex on this new subpath different from the end vertices $gy_1$ and $gy_2$, then
$\|\theta(v)\|_2<\|\theta(g)\|_2$. Recall that $\theta(t)=\pm e_i$, so all the vertices on the segment between 
$gy_1$ to $gy_1 t^{-r}$ differ only in the $i^{\rm th}$ coordinate. Moreover, by assumption $\theta(gy_1)_i\geq \theta(g)_i-1\geq R$
and $\theta(t)_i$ and $\theta(gy_1)_i$ have the same sign. Thus, if we require that $r\leq R$, then 
$|\theta(gy_1 t^{-j})_i|=|\theta(gy_1)_i|-j$ for $0\leq j\leq r$, so the $l^2$-norm strictly goes down 
as we move from $gy_1$ to $gy_1 t^{-r}$. Similarly, the $l^2$-norm strictly goes down when we move from $gy_2$ to $gy_2 t^{-r}$. 

Let us now consider the vertices of $p'$ between
$gt^{-r}$ and  $gy_2 t^{-r}$ (the vertices between $gy_1 t^{-r}$ and $g t^{-r}$ are treated similarly). Any such vertex is of the form $gt^{-r}v$ where $v$ is a prefix of the $(\chi,S)$-non-negative form of $t^r y_2 t^{-r}$ chosen above. 
Since $\|t^r y_2 t^{-r}\|_1\leq 2r+1$,
by construction we get
$$\|\theta(v)\|_1\leq (2BC+1)(AB+A+2r+1)+2A\leq (2BC+1)(AB+3A+2r+1)$$ and $\chi(v)\geq 0$, so $(\theta(g),\theta(v))=-\chi(v)\leq 0$. 
Also recall that $(\theta(g), \theta(t))\geq R$. 
Therefore we have
   \begin{align*}
	 \|\theta(gt^{-r}v)\|_2^2 &=\|\theta(g)\|_2^2+\|\theta(t^{-r}v)\|_2^2+2(\theta(g), \theta(t^{-r}))+
2(\theta(g), \theta(v)) \\
&\leq \|\theta(g)\|_2^2+\|\theta(t^{-r}v)\|_1^2-2r(\theta(g), \theta(t))+
2(\theta(g), \theta(v)) \\
	 &\leq \|\theta(g)\|_2^2 +(\|\theta(v)\|_1+r)^2-2rR \\
	 &\leq \|\theta(g)\|_2^2 + ((2BC+1)(AB+3A+2r+1)+r)^2-2rR\\
   &< \|\theta(g)\|_2^2 + ((2BC+2)(AB+3A+2r+1))^2-2rR.
\end{align*}
The first inequality above holds since for any $(a_1,\ldots,a_n)\in\dbZ^n$ we have 
$\|(a_1,\ldots,a_n)\|_2^2=\sum |a_i|^2\leq (\sum |a_i|)^2=\|(a_1,\ldots,a_n)\|_1^2$.
\vskip .1cm

 Recall that $R=16(BC+1)^2(AB+3A+3)$. For simplicity of notation write $P=16(BC+1)^2$ and $Q=\frac{AB+3A+1}{2}$,
 so that $R=2P(Q+1)$. If we set $r=\lceil Q \rceil$ (which satisfies the condition $r\leq R$ imposed earlier in the proof), 
then $Q\leq r<Q+1$ and hence $Q+r\leq 2r<2(Q+1)$. Therefore
$$((2BC+2)(AB+3A+2r+1))^2-2rR=P(Q+r)^2-P\cdot 2r\cdot 2(Q+1)<0$$
 which completes the proof.
 \end{proof}
 
\subsection{An explicit generating set for $\K$}

To finish the constructive proof of finite generation given in this section, we need to establish  an effective version of Theorem~\ref{preimage}.  
 
For a technical reason, in the following two results it will be convenient to work with left Cayley graphs (note that earlier in this section we worked with the commonly used right Cayley graphs). The left Cayley graph of a group $G$ with respect to $S$, denoted by $Cay_{left}(G,S)$, is defined in the same way as $Cay(G,S)$ except
that edges have the form $(g,s^{\pm 1}g)$ with $s\in S$. Note that $Cay_{left}(G,S)$ and $Cay(G,S)$ are isomorphic as graphs via the inversion map $g\mapsto g^{-1}$.

Theorem~\ref{lem:ReidSchreier} below is a variation of the Reidemeister-Schreier rewriting process.
This result is undoubtedly well known, but we are not aware of a specific reference in the literature, so we will provide a proof.

Recall that if $K$ is a subgroup of a group $G$, a left transversal of $K$ in $G$ is a subset $T$ of $G$ which contains exactly one element from each left coset of $K$.

\begin{Theorem}
\label{lem:ReidSchreier}
Let $G$ be a group generated by a set $S$, let $\K$ be a (not necessary normal) subgroup of $G$, and let $\theta:G\to G/\K$
be the natural projection. Let $\F\subseteq G/\K$ be such that $\theta^{-1}(\F)$ is connected in $Cay_{left}(G,S)$. Let $T$ be a left transversal for $\K$ in $G$, let $T_\F=T\cap \theta^{-1}(\F)$ (note that $|T_\F|=|\F|$) and let 
$$U=\{(s,t): t\in T_\F, s\in S,\, \theta(st)\in \F\}.$$
Then $\K$ is generated by the set $$S_\K=\{{\overline{st}\,}^{-1}st: (s,t)\in U\}$$
where for every $g\in G$ by $\overline g\in T$ we denote the unique
element of $T$ such that $\theta(g)=\theta(\overline g)$.
In particular, if $\F$ and $S$ are finite, then
$\K$ can be generated by (at most) $|\F||S|$ elements.
\end{Theorem}
\begin{Remark}\rm If we take $\F=G/\K$, then $\theta^{-1}(\F)=G$ is automatically connected. In this case $S_\K$ is the usual Reidemeister-Schreier generating set for $\K$.
\end{Remark}
\begin{proof}
Let $\K'$ be the subgroup generated by $S_\K$. Clearly 
$S_\K\subseteq \K$ (since $\overline{g}\K= g \K$ for all $g\in G$
by definition of $\overline{g}$), so $\K'\subseteq \K$.
Let us now prove that $\K\subseteq \K'$.

Take any $k\in \K$. Since $\K=\theta^{-1}(1)\subseteq\theta^{-1}(\F)$ and $\theta^{-1}(\F)$ is connected in $Cay_{left}(G,S)$, we can find a path
$1=y_0,y_1,\ldots, y_m=k$ in $Cay_{left}(G,S)$ with $\theta(y_i)\in \F$ for each $i$.
Since $T$ is a transversal for $\K$, for each $0\leq i<m$
we can uniquely write $y_i=t_i k_i$ where $t_i\in T$ and $k_i\in \K$.
Note that $k_0=1$ and $k_m=k$. Thus, to prove that $k\in \K'$ it suffices to show that $k_{i+1}k_i^{-1}\in \K'$ for each $0\leq i<m$. 

Let us fix $0\leq i<m$. By assumption $y_i$ and $y_{i+1}$ are connected by an edge in $Cay_{left}(G,S)$, so there exists $s\in S$ such that $y_{i+1}=sy_i$ or $y_i=sy_{i+1}$. 

First consider the case $y_{i+1}=sy_i$.
We have $\theta(t_i)=\theta(t_i k_i)=\theta(y_i)\in \F$, so $t_i\in T_\F$ and
$\theta(st_i)=\theta(st_i k_i)=\theta(sy_i)= \theta(y_{i+1})\in \F$, whence $(s,t_i)\in U$.

Also $t_{i+1}k_{i+1}=y_{i+1}=sy_i=st_i k_i$ and hence
$t_{i+1}=\overline{t_{i+1}k_{i+1}}=\overline{st_i k_i}=
\overline{st_i}$. Finally, the equality $t_{i+1}k_{i+1}=st_i k_i$ implies that
$k_{i+1}k_i^{-1}=t_{i+1}^{-1}st_i=\overline{st_i}^{-1}st_i\in S_\K$,
as desired.
\vskip .12cm
In the case $y_i=sy_{i+1}$ we can repeat the above argument swapping $i$ and $i+1$ in every expression and observe in the end
that $k_{i}k_{i+1}^{-1}\in \K'$ forces $k_{i+1}k_{i}^{-1}=(k_{i}k_{i+1}^{-1})^{-1}\in \K'$.
\end{proof}
 
\paragraph{\bf Making the generating set for $\K$ more explicit} Let us now consider the special case where $G/\K$ is abelian and torsion-free.
Let $\theta:G\to G/\K$ be the natural projection. Choose an ordered basis
$E=\{e_1,\ldots, e_n\}$ of $G/\K$, and use it to identify $G/\K$ with $\dbZ^n$.
Recall that by Lemma~\ref{lem:Nielsen}, $G$ has a finite generating set $S$ such that $\theta(S)=E$ or
$\theta(S)=E\cup\{0\}$. Note that $\theta$ restricted to $S$ need not be injective.

We will show that if $\F$ satisfies a certain technical condition (see the definition of a Schreier set below), one can obtain an even more explicit finite generating set for $\K$ (by slightly modifying the set $S_\K$ from Theorem~\ref{lem:ReidSchreier}).

Choose $s_1,\ldots, s_n\in S$ with $\theta(s_i)=e_i$, and let $S_1=\{s_1,\ldots, s_n\}$. Let
$S_2$ be the elements of $S\setminus S_1$ which lie outside of $\K$, and let $S_3$ be the elements of $S\setminus S_1$ which lie in $\K$. Thus $S=S_1\sqcup S_2\sqcup S_3$. For each $s\in S_2$ let $d(s)$ be the unique integer such that $\theta(s)=e_{d(s)}$.
 
\begin{Definition}\rm In the above notations, a subset $\F$ of $G/\K$ will be called {\it Schreier} if 
the following hold:
\begin{itemize}
\item[(i)] $(0,\ldots,0)\in\F$
\item[(ii)] for every $(a_1\ldots, a_n)\in \F$ and any path $p$ in $Cay(G/K,E)$ from $(0,\ldots,0)$ to
$(a_1\ldots, a_n)$ of minimal length, all vertices of $p$ lie in $\F$.
\end{itemize}
Equivalently, $\F\subseteq G/\K$ is Schreier if for every
$(a_1\ldots, a_n)\in \F$, any element $(b_1,\ldots, b_n)\in G/\K$ such that $|b_i|\leq |a_i|$ and $a_ib_i\geq 0$ for all $i$ also lies in $\F$.  
\end{Definition}
\begin{Remark}\rm It is clear that for any $p\in [1,\infty]$, the $l^p$-ball centered at $0$ is a Schreier set. 
 \end{Remark}
 
 Recall that for group elements $x,y$ we set $[x,y]=x^{-1}y^{-1}xy$ and $x^y=y^{-1}xy$.
 \begin{Theorem}
 \label{preimage_effective}
 Assume that $G,\K$ and $S$ satisfy the above conditions. Let $\F\subseteq G/\K$ be a Schreier subset of $G/\K$
 such that $\theta^{-1}(\F)$ is connected. Then $\K$ is generated by the following three types of elements:
 \begin{itemize}
 \item[(a)] $[s_i,s_j]^{s_i^{a_i}s_{i+1}^{a_{i+1}}\ldots\,\, s_n^{a_n}}$ where $1\leq i<j\leq n$ and 
$(\underbrace{0,\ldots,0}_{i-1\mbox{ \tiny times }},a_i,a_{i+1},\ldots, a_n)\in \F$
 \item[(b)] $(s^{-1} s_{d(s)})^{s_1^{a_1}s_{2}^{a_{2}}\ldots s_n^{a_n}}$ where $s\in S_2$ and $(a_1,\ldots, a_n)\in \F$
\item[(c)] $s^{s_1^{a_1}s_{2}^{a_{2}}\ldots s_n^{a_n}}$ where $s\in S_3$ and $(a_1,\ldots, a_n)\in \F$.
 \end{itemize}
 \end{Theorem}
 \begin{proof} Clearly all elements in (a)-(c) above lie in $\K$. Let $T=\{s_1^{a_1}\ldots s_n^{a_n}: a_i\in\dbZ\mbox{ for all }i\}$. By construction $T$ is a transversal for $\K$ in $G$, and let $S_\K$ be the corresponding set from Theorem~\ref{lem:ReidSchreier}. We need to show that every element of $S_\K$ can be expressed in terms of elements of type (a)-(c). Let us denote the subgroup generated by those elements by $\K'$.

So take
 any $s\in S$, $t\in T$ with $\theta(t)\in \F$. Thus $t=\prod\limits_{j=1}^n s_j^{a_j}$ for some $(a_1,\ldots, a_n)\in \F$. We are also allowed
 to assume that $\theta(st)\in \F$, but this extra condition will not be needed for the argument. We will consider 3 cases depending on which of the subsets $S_1, S_2$ and $S_3$ the generator $s$ lies in.
 \vskip .12cm
 {\it Case 1: $s\in S_1$, so $s=s_j$ for some $1\leq j\leq n$}. Let us write $t=uv$ where $u=\prod\limits_{i=1}^j s_i^{a_i}$ and
 $v=\prod\limits_{i=j+1}^n s_i^{a_i}$. Then $\overline{st}=\prod\limits_{i=1}^{j-1}s_i^{a_i}\cdot s_j^{a_j+1}\prod\limits_{i=j+1}^{n}s_i^{a_i}=us_jv$,
 so $${\overline{st}\,}^{-1}st=v^{-1}s_j^{-1}u^{-1}s_juv=[s_j,u]^v=([u,s_j]^v)^{-1}.$$
  Using the formulas $[xy,z]=[x,z]^y [y,z]$ and $[x^{-1},y]=([x,y]^{x^{-1}})^{-1}$ and the fact that $\F$ is Schreier, it is straightforward to express $[u,s_j]^v$ in terms of elements of type (a). Let us illustrate this in the
 case $j=3$, $u=s_1^2 s_2^{-1}$ (so that $v=\prod\limits_{i=3}^n s_i^{a_i}$ for some $a_i$).
 We have
\begin{equation}
\label{eq:Schreier}
[u,s_j]^v=[s_1^2 s_2^{-1},s_3]^v=[s_1^2,s_3]^{s_{2}^{-1}v}[s_2^{-1},s_3]^{v}
=[s_1,s_3]^{s_1 s_{2}^{-1}v}[s_1,s_3]^{s_{2}^{-1}v}([s_2,s_3]^{s_2^{-1} v})^{-1}.
\end{equation}
By assumption, $(2,-1,a_3,\ldots, a_n)\in\F$. Since $\F$ is Schreier, 
$(1,-1,a_3,\ldots, a_n)\in\F$ and $(0,-1,a_3,\ldots, a_n)\in\F$ as well. Thus,
all the factors in the last expression in \eqref{eq:Schreier} are generators of type (a), as desired.
\vskip .12cm

 {\it Case 2: $s\in S_2$}. Let $j=d(s)$. As in Case~1 we write $t=uv$ where $u=\prod\limits_{i=1}^j s_i^{a_i}$ and
 $v=\prod\limits_{i=j+1}^n s_i^{a_i}$, so $\overline{st}=us_jv$. Hence ${\overline{st}\,}^{-1}st=(us_jv)^{-1} suv$ and therefore
\begin{equation}
\label{eq:Schreier2}
({\overline{st}\,}^{-1}st)^{-1}({\overline{s_j t}\,}^{-1}s_j t)=(suv)^{-1}us_jv (us_jv)^{-1} s_juv=(s^{-1}s_j)^{uv}=(s^{-1}s_{d(s)})^t.
\end{equation}
The last expression in \eqref{eq:Schreier2} is an element of type (b). Since ${\overline{s_j t}\,}^{-1}s_j t$ lies in $\K'$ by Case~1, it follows from \eqref{eq:Schreier2} that ${\overline{st}}^{-1}st\in \K'$ as well.

 \vskip .12cm
 {\it Case 3: $s\in S_3$}. In this case $s\in \K$, so $\overline{st}=t$ and hence ${\overline{st}}^{-1}st=s^t$, which is an element of type (c).
 \end{proof}
 
\section{Effective finite generation of $[\IA_n,\IA_n]$}
\label{sec:IAn}
In this section we will prove Theorem~\ref{thm:main_IAn}.
Throughout this section we fix an integer $n\geq 2$ and let $[n]=\{1,2,\ldots,n\}$. 

Magnus~\cite{Ma} proved that $\IA_n$ is generated by the elements $K_{ij}$ with $i\neq j\in [n]$ and
$K_{ijk}$ with $i,j,k\in [n]$, $i,j,k$ distinct defined by
$$K_{ij}:\left\{\begin{array}{l}x_i\mapsto x_j^{-1}x_i x_j\\
x_k\mapsto x_k\mbox{ for }k\neq i\end{array}
\right.,\quad
K_{ijk}:\left\{\begin{array}{l}x_i\mapsto x_i[x_j,x_k]\\
x_l\mapsto x_l\mbox{ for }l\neq i
\end{array}\right.
.
$$
Clearly $K_{ikj}=K_{ijk}^{-1}$, so $\IA_n$ is generated by the set $\{K_{ij}\}\cup \{K_{ijk}: j<k\}$.

Throughout this section we set $S=\{K_{ij}\}\cup \{K_{ijk}: j<k\}$ and will refer to $S$
as the {\it Magnus generating set} for $\IA_n$.
\footnote{New, more geometric, proofs of the fact that $S$ generates $\IA_n$ were given in \cite{BBM} and \cite{DP}.
The proof given in \cite{Ma} has two parts: one first shows that $S$ generates $\IA_n$ as a normal
subgroup of $\Aut(F_n)$ and then shows that the subgroup generated by $S$ is normal in $\Aut(F_n)$. 
Both \cite{BBM} and \cite{DP} gave very different proofs for the first part, but followed the
original argument of Magnus for the second part.}
An easy computation shows that $|S|=n{n\choose 2}$.

The following commutation relations between the Magnus generators of $\IA_n$ are straightforward to check:
\begin{Lemma}
\label{lem:commrelations}
The following hold:
\begin{itemize}
\item[(a)] $[K_{ij},K_{kl}]=1$ if $i\not\in\{k,l\}$ and $k\not\in\{i,j\}$.
\item[(b)] $[K_{ij},K_{klm}]=1$ if $i\not\in\{k,l,m\}$ and $k\not\in\{i,j\}$.
\end{itemize}
In particular, two Magnus generators commute if their sets of indices are disjoint. 
\end{Lemma}

Since $\IA_n$ is normal in $\Aut(F_n)$ and $\Aut(F_n)/\IA_n\cong\GL_n(\dbZ)$, the abelianization 
$\IA_{n}^{\rm ab}\cong \IA_n/[\IA_n,\IA_n]$ has the natural structure of a $\GL_n(\dbZ)$-module. As a $\GL_n(\dbZ)$-module, $\IA_{n}^{ab}$ is canonically isomorphic to $V^*\otimes (V\wedge V)$ where $V=\dbZ^n$, considered as a standard $\GL_n(\dbZ)$-module, and $V^*=\Hom(V,\dbZ)$ is the dual module. This isomorphism was first established by Formanek~\cite{Fo}, 
but there are several alternative proofs in the literature (e.g. see \cite{DP}). 
\vskip .12cm

Let $e_1,\ldots, e_n$ be the standard basis of $V$, and let $e_1^*,\ldots, e_n^*$ be the dual basis. Given $x\in \IA_n$, let $[x]$ denote the image of $x$ in $\IA_n^{\rm ab}$. The above isomorphism $\IA_{n}^{ab}\cong V^*\otimes (V\wedge V)$ is given by
\begin{equation}
\label{eq:formanek}
[K_{ij}]\mapsto e_i^*\otimes (e_i\wedge e_j)\quad\mbox{ and }\quad [K_{ijk}]\mapsto e_i^*\otimes (e_j\wedge e_k).
\end{equation}
From now on we will identify $\IA_n^{\rm ab}$ with $V^*\otimes (V\wedge V)$ via the map \eqref{eq:formanek}.
\vskip .12cm

Let $N=n{n\choose 2}=\frac{n^2(n-1)}{2}$. By the above discussion, $\IA_n^{\rm ab}\cong \dbZ^N$ as abelian groups,
and moreover the natural projection $\IA_n\to \IA_n^{\rm ab}$ is injective on $S=\{K_{ij}\}\cup\{K_{ijk}: j<k\}$ and maps $S$ to
$E=\{e_i^*\otimes (e_j\wedge e_k): j<k\}$, which is a basis of $\IA_n^{\rm ab}$. In particular, $S$ satisfies the conclusion of Lemma~\ref{lem:Nielsen} for $G=\IA_n$ and $\K=[\IA_n,\IA_n]$.

Now let 
$$Z=\{K_{12}, K_{34}, K_{56},K_{78}\}.$$
By Lemma~\ref{lem:commrelations}, elements of $Z$ commute with each other, and it is easy to check that every element of $S$ commutes with an element of $Z$. These two properties immediately imply that the pair $(S,Z)$ is chain-centralizing. We now need to construct $\Phi$ satisfying the Regularity Hypothesis, but first we make some general observations.
\medskip

\paragraph{\bf Action on the space of characters} For any group $G$ we have a natural action of $\Aut(G)$ on the space of characters $\Hom(G,\dbR)$ given by
$$(\phi \chi)(x)=\chi(\phi^{-1}(x))\mbox{ for any } x\in G \mbox{ and } \phi\in\Aut(G).$$ 
Clearly, $\Inn(G)$ acts trivially, so we get an action of $\Out(G)=\Aut(G)/\Inn(G)$.

Next note the centralizer of $\IA_n$ in $\Aut(F_n)$ is trivial (since already the centralizer of $\Inn(F_n)$ in $\Aut(F_n)$ is trivial) and hence the conjugation action of $\Aut(F_n)$ on $\IA_n$ yields an embedding of $\Aut(F_n)$ into $\Aut(\IA_n)$.
\footnote{In fact, it has been recently proved in \cite{BW} that this embedding is an isomorphism for $n\geq 3$, but this is not essential for our purposes.}

This, in turn, induces an embedding of $\Aut(F_n)/\IA_n\cong\GL_n(\dbZ)$ into $\Out(\IA_n)$ and thereby an action
of $\GL_n(\dbZ)$ on $\Hom(\IA_n,\dbR)$. It is easy to check that this is a ``standard'' action, dual to the action of
$\GL_n(\dbZ)$ on $\IA_n^{\rm ab}$ discussed above, but it is important for us that it comes from an action of $\Aut(F_n)$ on $\IA_n$.

The key technical result that we will prove in this section is the following lemma:

 \begin{Lemma}
 \label{lemma:SL_n}
There exists a finite subset $\Omega$ of $\GL_n(\dbZ)$ with the following properties:
\begin{itemize}
\item[(a)] For any nonzero character $\chi$ of $G=\IA_n$ there is $g\in \Omega$ such that
$|g\chi(z)|\geq \frac{M(\chi)}{3}$ for all $z\in Z$.
\item[(b)] Every $g\in\Omega$ admits a lift $\phi\in\Aut(F_n)$ with $B(\phi^{\pm 1})\leq 150$ and $A(\phi^{\pm 1})\leq 8100$.
\end{itemize}
\end{Lemma}
\begin{Remark}\rm Recall that $M(\chi)=\max\{|\chi(s)|: s\in S\}$ and that throughout this section $S$ is the Magnus generating set of $\IA_n$.
\end{Remark}

As a consequence of Lemma~\ref{lemma:SL_n} we deduce that the pair $(G,\K)=(\IA_n,[\IA_n,\IA_n])$
satisfies the Regularity Hypothesis with explicit constants:

\begin{Corollary} 
\label{cor:reghypIAn}
Let $G=\IA_n$ and $\K=[\IA_n,\IA_n]$. Then the Regularity Hypothesis holds for $C=3$ and some finite $\Phi$ satisfying 
$B(\Phi\cup\Phi^{-1})\leq 150$ and $A(\Phi\cup\Phi^{-1})\leq 8100$. 
\end{Corollary}
\begin{proof} Let $\Phi$ be the set of all elements $\phi^{-1}$ where $\phi$ ranges over all lifts from the conclusion of
Lemma~\ref{lemma:SL_n}(b). Then $\Phi$ is finite (since $\Omega$ is finite) and $B(\Phi\cup\Phi^{-1})\leq 150$ and 
$A(\Phi\cup\Phi^{-1})\leq 8100$
by Lemma~\ref{lemma:SL_n}(b). 

Now let $\chi$ be any nonzero character of $G$. By Lemma~\ref{lemma:SL_n}(a) there exists $g\in\Omega$ such that
$|g\chi(z)|\geq \frac{M(\chi)}{3}$ for all $z\in Z$. Let $\phi\in \Aut(F_n)$ be the lift of $g$ from Lemma~\ref{lemma:SL_n}(b).
Since $\chi(\phi^{-1}(x))=(\phi \chi)(x)=g\chi(x)$ for all $x\in G$ and $\phi^{-1}\in\Phi$ by construction, 
Regularity Hypothesis holds for this $\Phi$ and $C=3$.
\end{proof}

The proof of Lemma~\ref{lemma:SL_n} will consist of two parts. First we will construct $g$ satisfying (a). This will be done in several steps, and $g$
will be constructed as a product of at most 9 unit transvections and at most 2 permutation matrices (this ensures that there are only
finitely many possibilities for $g$). Then we will prove (b) using the specific form of $g$ constructed in the proof of (a).

In the computations below it will be convenient to use the following notation: for a character $\lambda$ of $G$ and
$i,j,k\in [n]$ we set
$$c_{ijk}(\lambda)=\lambda(e_i^{*}\otimes (e_j\wedge e_k)).$$
Note that we can reformulate the condition on $g$ in Lemma~\ref{lemma:SL_n}(a) in terms of the coefficients $c_{ijk}$ as follows:
\begin{equation}
\label{eq:SLn}
|c_{iij}(g\chi)|\geq \frac{M(\chi)}{3}\mbox{ for }(i,j)\in\{(1,2),(3,4),(5,6),(7,8)\}. 
\end{equation}
\vskip .12cm
Given $i,j\in [n]$ with $i\neq j$ and a permutation $\sigma$ of $[n]$, define $E_{ij}, F_{\sigma}\in \GL_n(\dbZ)$ by
\begin{equation}
\label{eij_def}
E_{ij}:\left\{\begin{array}{l}e_j\mapsto e_j+ e_i\\
e_k\mapsto e_k\mbox{ for }k\neq j\end{array},
\right.\qquad
F_{\sigma}: e_k\mapsto e_{\sigma(k)} \mbox{ for all }k.
\end{equation}

\begin{proof}[Proof of Lemma~\ref{lemma:SL_n}(a)]

In each step below $M$ will denote a positive real number and $\lambda$ will denote an arbitrary character of $G$ (which will vary from step to step).
   \medskip
   
   \underline{Step 1:} If $M(\lambda)\geq M$, there is a permutation matrix $g_1$ such that
   $|c_{112}(g_1{\lambda})|\geq M$ or $|c_{132}(g_1 {\lambda})|\geq M$.

\medskip
By definition of $M(\lambda)$ we have $|c_{iij}(\lambda)|\geq M$ or $|c_{ijk}(\lambda)|\geq M$ for some distinct $i,j,k$. The result now follows from the obvious fact that $c_{ijk}(F_{\sigma}\lam)=c_{\sigma^{-1}(i)\sigma^{-1}(j)\sigma^{-1}(k)}(\lam)$.

   \medskip
   
   In the next step we give different arguments depending on which case occurred in Step~1.
   
   \medskip
   
   \underline{Step 2A:} If $|c_{132}({\lambda})|\geq M$ and $|c_{112}(\lambda)|<M$, there exists $g_2=E_{31}^{\pm 2}$ such that
   $|c_{112}(g_2\lambda)|\geq M$ and $|c_{332}(g_2\lambda)|\geq M$.
\medskip

   We have 
\begin{align*}&c_{112}(E_{31}^{\pm 2}\lambda)=c_{112}(\lambda)\mp 2 c_{132}(\lambda)& &c_{332}(E_{31}^{\pm 2}\lambda)=c_{332}(\lambda)\pm 2 c_{132}(\lambda).&
\end{align*}
   
By assumption $|c_{112}(\lambda)\mp 2 c_{132}(\lambda)|\geq M$ for any choice of sign and 
$|c_{332}(\lambda)\pm 2 c_{132}(\lambda)|\geq M$ for some choice of sign, so either $E_{31}^2$ or $E_{31}^{-2}$ can be used as $g_2$.

   \medskip
   
   \underline{Step 2B:} If $|c_{112}(\lambda)|\geq M$, there exists $g\in \GL_n(\dbZ)$ which is either the identity matrix or 
$E_{13}^{\eps} E_{31}^{\eps}$ with $\eps=\pm 1$ such that  $|c_{112}(g_2 \lambda)|\geq M$ and $|c_{332}(g_2 \lambda)|\geq \frac{M}{3}$.
   
   \medskip
   
   If $|c_{332}(\lambda)|\geq \frac{M}{3}$, then $g_2=1$ obviously works, so assume that $|c_{332}(\lambda)|< \frac{M}{3}$.
For $\eps=\pm 1$ by direct computation we have 
   \begin{align*}
     c_{112}(E_{13}^{\eps} E_{31}^{\eps} \lambda) &= 2c_{112}(\lambda)-c_{332}(\lambda)+\eps(c_{312}(\lambda)-2c_{132}(\lambda))\\
     c_{332}(E_{13}^{\eps} E_{31}^{\eps} \lambda) &= 2c_{332}(\lambda)-c_{112}(\lambda)-\eps(c_{312}(\lambda)-2c_{132}(\lambda)).
\end{align*}
	    Since $|c_{332}(\lambda)|< \frac{M}{3}$ and $|c_{112}(\lambda)|>M$, we have $|2c_{112}(\lambda)-c_{332}(\lambda)|>\frac{5M}{3}$ and $|2c_{332}(\lambda)-c_{112}(\lambda)|>\frac{M}{3}$; moreover,
$2c_{112}(\lambda)-c_{332}(\lambda)$ and $2c_{332}(\lambda)-c_{112}(\lambda)$ have different signs. Therefore, choosing $\eps=\pm 1$ such that 	$2c_{112}(\lambda)-c_{332}(\lambda)$ and $\eps(c_{312}(\lambda)-2c_{132}(\lambda))$ have the same sign, we obtain the desired $g_2$.

   \medskip
   
   \underline{Step 3:} If $|c_{112}(\lambda)|\geq M$ and $|c_{332}(\lambda)|\geq \frac{M}{3}$, there exists $g_3=E_{24}^{\pm 1}$  such that $|c_{112}(g_3\lambda)|\geq M$ and $|c_{334}(g_3\lambda)|\geq \frac{M}{3}$.
   \medskip
   
   This is clear since $c_{112}(E_{24}^{\pm 1}\lambda)=c_{112}(\lambda)$ and
   $c_{334}(E_{24}^{\pm 1}\lambda)=c_{334}(\lambda)\mp c_{332}(\lambda)$.

   \medskip
   
   \underline{Step 4:} If $|c_{112}(\lambda)|\geq M$ and $|c_{334}(\lambda)|\geq \frac{M}{3}$, there exists $g_4$ equal to
   either $E_{26}^{\pm 1}$ or $E_{26}^{\pm 1}E_{15}^{\eps}E_{51}^{\eps}$ for $\eps=\pm 1$ such that
$|c_{112}(g_4\lambda)|\geq M$ and
$|c_{iij}(g_4\lambda)|\geq \frac{M}{3}$ for $(i,j)=(3,4)$ and $(5,6)$.

To do this we apply Steps~2 and 3 with indices $3$ and $4$ replaced by $5$ and $6$. Since we will be acting by matrices
$E_{ij}$ with $i,j\not\in\{3,4\}$, the value of $c_{334}$ will not change. 

   \medskip
   
   \underline{Step 5:} If $|c_{112}(\lambda)|\geq M$ and
$|c_{iij}(\lambda)|\geq \frac{M}{3}$ for $(i,j)=(3,4)$ and $(5,6)$, there exists $g_5$ equal to
   either $E_{28}^{\pm 1}$ or $E_{28}^{\pm 1}E_{17}^{\eps}E_{71}^{\eps}$ for $\eps=\pm 1$ such that
$|c_{112}(g_5\lambda)|\geq M$ and $|c_{iij}(g_5\lambda)|\geq \frac{M}{3}$ for $(i,j)=(3,4), (5,6)$ and $(7,8)$.

The argument here is identical to Step 4.   
   \medskip
   
 Putting all the steps together, we obtain the desired $g\in \GL_n(\dbZ)$ given by the product $g_5 g_4 g_3 g_2 g_1$ (recall that
$g_i$ is the matrix we acted by in Step~$i$).
\end{proof}

Before turning to the proof of Lemma~\ref{lemma:SL_n}(b), we will first explain how the lifts from the
conclusion of Lemma~\ref{lemma:SL_n}(b) will be constructed and derive some general bounds on the constants $A$ and $B$ for certain maps.
\medskip

\paragraph {\bf Constructing lifts} Since $g$ in the proof of Lemma~\ref{lemma:SL_n}(a) is explicitly constructed as a product, we can obtain a lift of $g$ by simply lifting each factor. The natural lift of a transposition matrix $F_{ij}$ is $\widetilde {F_{ij}}\in \Aut(F_n)$ which swaps $x_i$ and $x_j$ and fixes other generators. A transvection matrix $E_{ij}$ can be lifted to either left or right Nielsen map $R_{ji}$ or $L_{ji}$ defined by
$$R_{ji}:\left\{\begin{array}{l}x_j\mapsto x_jx_i\\
x_k\mapsto x_k\mbox{ for }k\neq j\end{array}
\right.,\quad
L_{ji}:\left\{\begin{array}{l}x_j\mapsto x_i x_j\\
x_l\mapsto x_l\mbox{ for }l\neq i
\end{array}\right.
.
$$
Note that we can use different lifts for different occurrences of $E_{ij}$ but this does not seem to matter for the resulting bound. 

It is clear that
$A(\widetilde {F_{ij}})=B(\widetilde {F_{ij}})=1$, and an explicit computation in \cite{DP}
(see Table~1 in \cite[Appendix~A]{DP}) shows that $A(R_{ij}^{\pm1})=A(L_{ij}^{\pm 1})=6$
and $B(R_{ij}^{\pm1})=L(R_{ij}^{\pm1})=4$. Note that $g$ in the proof of Lemma~\ref{lemma:SL_n}(a) is a product of at most 2 transpositions
and at most 9 unit transvections. Combining these facts with the easy observation that 
\begin{equation}
\label{eq:AB}
A(\phi\psi)\leq A(\phi)A(\psi) \mbox{ and } 
B(\phi\psi)\leq B(\phi)B(\psi)
\end{equation}
for any $\phi,\psi\in \Aut(F_n)$, 
we already deduce that $g$ has a lift $\phi$
with $A(\phi^{\pm 1})\leq 6^{9}$ and $B(\phi^{\pm 1})\leq 4^{9}$.

To improve those bounds, we will prove the following general lemma:

\begin{Lemma}
\label{lem:AB}
Let $G,\K$ and $S$ be as in Lemma~\ref{lem:Nielsen}.
The following hold:
\begin{itemize}
\item[(a)] For any $\phi,\psi\in \Aut(G)$ we have $A(\phi\psi)\leq A(\psi)B(\phi)+A(\phi)$.
\item[(b)] Suppose that $\phi_1,\ldots,\phi_k\in\Aut(G)$ with $A(\phi)\leq A$ for each $i$. For $1\leq m\leq k$
define $\psi_m=\prod_{i=1}^{m-1} \phi_i$ (by convention $\psi_1=1$). Then
$$A(\phi_1\ldots\phi_k)\leq A\sum\limits_{m=1}^{k}B(\psi_m).$$
\end{itemize}
\end{Lemma}
\begin{proof}(b) follows from (a) by straightforward induction, so we will only prove (a).
Fix $s\in S^{\pm 1}$. By definition of $A(\psi)$ and Observation~\ref{obs:theta}, there
exists an $S$-word $s_1\ldots s_r$ representing $\psi(s)$ such that for all $1\leq j\leq r$,
one can write $\theta(s_1\ldots s_j)$ as a product of at most $A(\psi)$ elements $\theta(s)$, $s\in S^{\pm 1}$.
Next for each $s\in S^{\pm 1}$ choose an $S$-word $w_s$ representing $\phi(s)$ such that
$\|\theta(u)\|_1\leq A(\phi)$ for every $S$-prefix $u$ of $w_s$.

Consider the $S$-word $w=w_{s_1}\ldots w_{s_r}$. Then $w$ represents $\phi\psi(s)$, and any $S$-prefix of
$w$ is equal to $w_{s_1}\ldots w_{s_{j-1}}u$ for some $1\leq j\leq r$ and $S$-prefix $u$ of $w_{s_j}$. 
By assumption, $\theta(w_{s_1}\ldots w_{s_{j-1}})$ is the product of at most $A(\psi)$ elements $\theta(w_s)$, $s\in S^{\pm 1}$.
Since $\|\theta(w_s)\|\leq B(\phi)$ by definition of $B(\phi)$ and $\|\theta(u)\|_1\leq A(\phi)$, we have
$$\|\theta(w_{s_1}\ldots w_{s_{j-1}}u)\|_1\leq \|\theta(w_{s_1}\ldots w_{s_{j-1}})\|_1+\|\theta(u)\|_1\leq A(\psi)B(\phi)+A(\phi).$$
Thus, $A(\phi\psi)\leq A(\psi)B(\phi)+A(\phi)$, as desired.
\end{proof}

We can now finish the proof of Lemma~\ref{lemma:SL_n} 
 \begin{proof}[Proof of Lemma~\ref{lemma:SL_n}(b)] First note that for $\phi\in \Aut(F_n)$, the constant 
 $B(\phi)$ depends only on the image of $\phi$ in $\GL_n(\dbZ)$, so we can talk about $B(g)$ for $g\in \GL_n(\dbZ)$.

We will consider the case when $g$ constructed in the proof of part (a) is equal to $E_{28} E_{17} E_{71} E_{26} E_{15} E_{51}E_{24} E_{13} E_{31}$. It is not hard to check that this $g$ represents the worst-case scenario. It is also easy to see that
$B(g)\geq B(h)$ for any prefix $h$ of $g$ and the same is true for $g^{-1}$.

To estimate $B(g^{-1})$ we first compute the action of $g$ on the basis elements of $V$ and $V^*$. We have
\begin{align*}
&ge_1=8e_1+4e_7+2e_5+e_3& &ge_2=e_2&      &ge_1^*=e_1^*-e_3^*-e_5^*-e_7^*& &ge_2^*=e_2^*-e_4^*-e_6^*-e_8^*&\\
&ge_3=4e_1+2e_7+e_5+e_3&  &ge_4=e_4+e_2& &ge_3^*=2e_3^*-e_1^*+e_5^*+e_7^*& &ge_4^*=e_4^*&\\
&ge_5=2e_1+e_7+e_5&       &ge_6=e_6+e_2& &ge_5^*=2e_5^*-e_1^*+e_7^*& &ge_6^*=e_6^*&\\
&ge_7=e_1+e_7&            &ge_8=e_8+e_2& &ge_7^*=2e_7^*-e_1^*+e_7^*& &ge_8^*=e_8^*.&
\end{align*}
Now it is easy to see that  $\|g(e_i^*\otimes (e_j\wedge e_k))\|_1$ takes on its maximal value when $i=3$, $j=1$ and
$k=4,6$ or $8$, and the maximum is equal to $5\cdot 15\cdot 2 =150$.
It follows that $B(g^{-1})\leq 150$ and similarly $B(g)\leq 150$. 

Now let $\phi$ be the lift of $g$ or $g^{-1}$ defined above. By construction, we can write $\phi=\prod\limits_{i=1}^k \phi_i$
where $k\leq 9$ and $A(\phi_i)\leq 6$ for all $i$. We just argued that $B\left(\prod\limits_{i=1}^j \phi_i\right)\leq 150$ for all $1\leq j\leq k$.
Hence Lemma~\ref{lem:AB}(b) yields $A(\phi)=A\left(\prod\limits_{i=1}^k \phi_i\right)\leq 6k\cdot 150\leq 6\cdot 9\cdot 150=8100$, as desired.
\end{proof}

We are finally ready to prove Theorem~\ref{thm:main_IAn}. In fact, we will prove a slightly stronger statement:

\begin{Theorem}
\label{thm:main_IAn2} Let $n\geq 8$. Let $N=n{n\choose 2}$, and let $S=\{s_1,\ldots, s_N\}$ be the Magnus generating set for $\IA_n$. Then $[\IA_n,\IA_n]$ is generated by elements of the form
$$[s_i, s_j]^{s_i^{a_i}s_{i+1}^{a_{i+1}}\ldots s_N^{a_N}}
\mbox{ where } 1\leq i<j\leq N \mbox{ and } 0\leq |a_m|< 5\cdot 10^{12} \mbox{ for each } m. \eqno (***)$$ 

In fact, $[\IA_n,\IA_n]$ is generated by elements of this form with the additional property that
$|\{m: a_m \neq 0\}|\leq 8n^2$.
\end{Theorem}
\begin{proof}
We start by proving that $[\IA_n,\IA_n]$ is generated by all elements of the form (***), which is precisely the assertion of 
Theorem~\ref{thm:main_IAn}.

First we compute the radius $R$ in Theorem~\ref{prop:connectedradius} applied to $G=\IA_n$ and $K=[\IA_n,\IA_n]$. 
Corollary~\ref{cor:reghypIAn} shows that
for a suitable $\Phi$ we have $B\leq 150$, $A\leq 8100$ and can take $C=3$. Hence the preimage of the $l^{\infty}$-ball of radius $$R=16(BC+1)^2(AB+3A+3)\leq 16\cdot 451^2\cdot (8100\cdot 153+3)< 5\cdot 10^{12}$$
is connected. Now Theorem~\ref{thm:main_IAn} follows directly from Theorem~\ref{preimage_effective} (note that for
$G=\IA_n$ we have $S=S_1$ in the notations from that theorem).

Let us now prove the full statement of Theorem~\ref{thm:main_IAn2}. As before, we identify $\IA_n^{ab}$ with $\dbZ^N$
via the map $$\prod_{i=1}^N s_i^{a_i}\,[\IA_n,\IA_n]\mapsto (a_1,\ldots, a_N).$$ For a point $(a_1,\ldots, a_N)\in \dbZ^N$
define its support as $$\supp((a_1,\ldots, a_N))=|\{m: a_m \neq 0\}|.$$ Let $\F$ be the set of all $(a_1,\ldots, a_N)\in \dbZ^N$ 
with $|a_i|\leq R$ for each $i$ and $|\supp((a_1,\ldots, a_N))|\leq 8n^2$. Clearly, $\F$ is a Schreier set. Thus,
Theorem~\ref{preimage_effective} reduces Theorem~\ref{thm:main_IAn2} to showing that $\theta^{-1}(\F)$ is connected
(where $\theta:\IA_n\to\dbZ^N$ is the natural projection).

So take any $x,y\in \IA_n$ with $\theta(x),\theta(y)\in \F$. We already know that $x$ and $y$ can be connected by a path $p$ which lies in the $\theta$-preimage of the $l^{\infty}$-ball of radius $R$. If $\theta(g)\in \F$ for every vertex $g$ of $p$, we are done;
otherwise, consider all vertices $g$ on $p$ such that $|\supp(\theta(g))|$ is largest possible, and among these vertices (if there is more than one), choose one where $\|\theta(g)\|_1$ is maximal.
In particular, by assumption $|\supp(\theta(g))|> 8n^2$ and thus $\theta(g)\neq \theta(x),\theta(y)$.

The vertices of $p$ which precede and succeed $g$ have the form $gy_1$ and $gy_2$ for some $y_1,y_2\in S^{\pm 1}$. 
Write $\theta(g)=(a_1,\ldots,a_N)$. Since $|\supp(\theta(g))|> 8n^2$, there exist more than $8n^2$ indices $m$ such that $a_m\neq 0$, and an easy calculation using Lemma~\ref{lem:commrelations} shows that there exists $1\leq m\leq N$ such that $a_m\neq 0$,
$s_m\neq y_1^{\pm 1},y_2^{\pm 1}$ and $s_m$ commutes with both $y_1$ and $y_2$. Without loss of generality, we can assume that $a_m>0$. Now
modify $p$ replacing the segment
$(gy_1,g,gy_2)$ of $p$ by $$(gy_1, gy_1s_m^{-1}=gs_m^{-1}y_1, gs_m^{-1},gs_m^{-1}y_2=gy_2s_m^{-1},gy_2).$$ Another easy calculation shows that for every vertex $v$ on this segment we have $|\supp(\theta(v))|\leq |\supp(\theta(g))|$ and
$\|\theta(v)\|_1<\|\theta(g)\|_1$. Thus, after applying such modification finitely many times, we will obtain a path connecting 
$x$ to $y$ which lies in $\theta^{-1}(\F)$, as desired. 
\end{proof}

\section{Effective finite generation of $[\calI_n^1, \calI_n^1]$ and the Johnson kernel}

\subsection{Preliminaries}
\label{sec:prelim}
Throughout this section we fix an integer $n\geq 0$ and let $\Sigma=\Sigma_n^1$ be an orientable surface of genus $n$ with $1$ boundary component. We start by introducing some curves and subsurfaces on $\Sigma$ that will be used throughout the proof.

First, it will be convenient to think of $\Sigma$ as a (closed) disk with $n$ handles attached; let us number the handles from $1$ to $n$.
We also fix a point $p_0$ on the boundary $\partial\Sigma$ (it will serve as the base point for all the fundamental groups considered below).

For each $1\leq i\leq n$, choose a point $p_i$ on the $i^{\rm th}$ handle and curves $a_i$ and $b_i$ passing through $p_i$ as shown on Figure~\ref{fig:surface}. Also choose (oriented) paths $\gamma_i$ from $p_0$ to $p_i$ as in Figure~\ref{fig:surface}. In particular, we require the sets 
$\gamma_i\setminus\{p_0\}$ to be disjoint.

\begin{figure}
\usetikzlibrary{arrows.meta}

\begin{tikzpicture}
  [>={Stealth[scale=0.75]}, fil/.style={circle, fill=black, inner sep=0pt, outer sep=0pt, minimum size=1.5mm, thick}]

  %surface boundary parameters
  \def\xstart{0 }; %start and end of straight line part of surface
  \def\xend{10 };
  \def\xmid{0.5*\xstart + 0.5*\xend };
  \def\yb{3 }; %height of surface boundary
 
  %handle size parameters

  \def\y{2.25 }; %y-coord of handle hole centers
  \def\xr{0.5 }; %radii of handle holes
  \def\yr{0.1 };
  
  \def\hi{0.75 }; %vertical radius of inner part of handle
  \def\ho{2.25 }; %vertical radius of outer part of handle
  
  %handle position parameters
   %first handle
   \def\xa{0.5 }; %center of left hole of first handle
   \def\xb{2 }; %center of right hole of first handle
  
   %second handle
   \def\xc{3.5 };
   \def\xd{5 }; 
  
   %last handle
   \def\xg{8 };
   \def\xh{9.5 };

  %draws surface boundary
  \draw[thick] (\xstart,0) -- (\xend,0);
  \draw[thick] (2.5,\yb) -- (3, \yb) (5.5,\yb) -- (7.5,\yb);
  \draw[thick] (\xstart,\yb) arc (90:270:1 and 0.5*\yb);
  \draw[thick] (\xend,\yb) arc (90:-90:1 and 0.5*\yb);
  
  \node[fil, scale=0.75] at (6, 0.5+\yb){};
  \node[fil, scale=0.75] at (6.5, 0.5+\yb){};
  \node[fil, scale=0.75] at (7, 0.5+\yb){};
  
  \node[fil, label=below:$p_0$] (p0) at (\xmid, 0){};

  %First Handle
  
  %draws left hole of handle
  \draw[dashed] (\xa+\xr, \y) arc (0:180: \xr and \yr);
  \draw[thick] (\xa+\xr, \y) arc (0:-180: \xr and \yr);
  
  %draws right hole of handle
  \draw[dashed] (\xb+\xr, \y) arc (0:180: \xr and \yr);
  \draw[thick] (\xb+\xr, \y) arc (0:-180: \xr and \yr);
  
  %draws handle part of handle
  \draw[thick] (\xb - \xr, \y) arc (0:180: 0.5*\xb-0.5*\xa-\xr and \hi);
  \draw[thick] (\xb + \xr, \y) arc (0:180: 0.5*\xb-0.5*\xa + \xr and \ho);

  %draws a_1
  \draw[thick, ->] (0.5*\xa + 0.5*\xb, \y+\ho) arc (90:210: 0.25*\xr and 0.5*\ho-0.5*\hi) node[pos=0, above]{$a_1$} node[pos=0.76, fil, label=180:$p_1$](p1){};
  \draw[thick] (0.5*\xa + 0.5*\xb, \y+\hi) arc (270:206: 0.25*\xr and 0.5*\ho-0.5*\hi);
  \draw[dashed] (0.5*\xa + 0.5*\xb, \y+\ho) arc (90:-90: 0.25*\xr and 0.5*\ho-0.5*\hi);
  
  %draws b_1
  \draw[thick] (\xb, \y) arc (0:68:0.5*\xb - 0.5*\xa and 0.5*\ho + 0.5*\hi);
  \draw[thick, ->] (\xa, \y) arc (180:65:0.5*\xb - 0.5*\xa and 0.5*\ho + 0.5*\hi);
  
  \draw[thick] (\xb, \y) arc (0:-180:0.5*\xb - 0.5*\xa and 0.5) node[midway, below]{$b_1$};

  %draws gamma_1
  \draw[thick, ->] (p0) arc (270:210:\xmid-\xb-0.5*\xr and \y) node[pos=0.5, left=3pt]{$\gamma_1$};
  \draw[thick] (\xb+0.5*\xr, \y) arc (180:212:\xmid-\xb-0.5*\xr and \y);
  \draw[thick] (\xb + 0.5*\xr, \y) arc (0:90: 0.5*\xb-0.5*\xa and 0.8*\ho);
  \draw[thick] (0.5*\xa + 0.5*\xb + 0.5*\xr, \y + 0.8*\ho) to[out=180, in=60] (p1);

  %Second Handle
  
  %draws left hole of handle
  \draw[dashed] (\xc+\xr, \y) arc (0:180: \xr and \yr);
  \draw[thick] (\xc+\xr, \y) arc (0:-180: \xr and \yr);
  
  %draws right hole of handle
  \draw[dashed] (\xd+\xr, \y) arc (0:180: \xr and \yr);
  \draw[thick] (\xd+\xr, \y) arc (0:-180: \xr and \yr);
  
  %draws handle part of handle
  \draw[thick] (\xd - \xr, \y) arc (0:180: 0.5*\xd-0.5*\xc-\xr and \hi);
  \draw[thick] (\xd + \xr, \y) arc (0:180: 0.5*\xd-0.5*\xc + \xr and \ho);

  %draws a_1
  \draw[thick, ->] (0.5*\xc + 0.5*\xd, \y+\ho) arc (90:210: 0.25*\xr and 0.5*\ho-0.5*\hi) node[pos=0, above]{$a_2$} node[pos=0.76, fil, label=180:$p_2$](p2){};
  \draw[thick] (0.5*\xc + 0.5*\xd, \y+\hi) arc (270:206: 0.25*\xr and 0.5*\ho-0.5*\hi);
  \draw[dashed] (0.5*\xc + 0.5*\xd, \y+\ho) arc (90:-90: 0.25*\xr and 0.5*\ho-0.5*\hi);
  
  %draws b_1
  \draw[thick] (\xd, \y) arc (0:68:0.5*\xd - 0.5*\xc and 0.5*\ho + 0.5*\hi);
  \draw[thick, ->] (\xc, \y) arc (180:65:0.5*\xd - 0.5*\xc and 0.5*\ho + 0.5*\hi);
  
  \draw[thick] (\xd, \y) arc (0:-180:0.5*\xd - 0.5*\xc and 0.5) node[midway, below]{$b_2$};

  %draws gamma_1
  \draw[thick, ->] (p0) arc (270:210:\xmid-\xd-0.5*\xr and \y) node[pos=0.75, left]{$\gamma_2$};
  \draw[thick] (\xd+0.5*\xr, \y) arc (180:212:\xmid-\xd-0.5*\xr and \y);
  \draw[thick] (\xd + 0.5*\xr, \y) arc (0:90: 0.5*\xd-0.5*\xc and 0.8*\ho);
  \draw[thick] (0.5*\xc + 0.5*\xd + 0.5*\xr, \y + 0.8*\ho) to[out=180, in=60] (p2);

  %Last Handle
  
  %draws left hole of handle
  \draw[dashed] (\xg+\xr, \y) arc (0:180: \xr and \yr);
  \draw[thick] (\xg+\xr, \y) arc (0:-180: \xr and \yr);
  
  %draws right hole of handle
  \draw[dashed] (\xh+\xr, \y) arc (0:180: \xr and \yr);
  \draw[thick] (\xh+\xr, \y) arc (0:-180: \xr and \yr);
  
  %draws handle part of handle
  \draw[thick] (\xh - \xr, \y) arc (0:180: 0.5*\xh-0.5*\xg-\xr and \hi);
  \draw[thick] (\xh + \xr, \y) arc (0:180: 0.5*\xh-0.5*\xg + \xr and \ho);

  %draws a_n
  \draw[thick, ->] (0.5*\xg + 0.5*\xh, \y+\ho) arc (90:210: 0.25*\xr and 0.5*\ho-0.5*\hi) node[pos=0, above]{$a_n$} node[pos=0.76, fil, label=180:$p_n$](pn){};
  \draw[thick] (0.5*\xg + 0.5*\xh, \y+\hi) arc (270:206: 0.25*\xr and 0.5*\ho-0.5*\hi);
  \draw[dashed] (0.5*\xg + 0.5*\xh, \y+\ho) arc (90:-90: 0.25*\xr and 0.5*\ho-0.5*\hi);
  
  %draws b_n
  \draw[thick] (\xh, \y) arc (0:68:0.5*\xh - 0.5*\xg and 0.5*\ho + 0.5*\hi);
  \draw[thick, ->] (\xg, \y) arc (180:65:0.5*\xh - 0.5*\xg and 0.5*\ho + 0.5*\hi);
  
  \draw[thick] (\xh, \y) arc (0:-180:0.5*\xh - 0.5*\xg and 0.5) node[midway, below]{$b_n$};

  %draws gamma_n
  \draw[thick, ->] (p0) arc (270:210:\xmid-\xh-0.5*\xr and \y) node[pos=0.5, above]{$\gamma_n$};
  \draw[thick] (\xh+0.5*\xr, \y) arc (180:212:\xmid-\xh-0.5*\xr and \y);
  \draw[thick] (\xh + 0.5*\xr, \y) arc (0:90: 0.5*\xh-0.5*\xg and 0.8*\ho);
  \draw[thick] (0.5*\xg + 0.5*\xh + 0.5*\xr, \y + 0.8*\ho) to[out=180, in=60] (pn);
  
\end{tikzpicture}
\caption{}
\label{fig:surface}
\end{figure}

Next define the curves $\alpha_i$ and $\beta_i$
by $\alpha_i=\gamma_i a_i \gamma_i^{-1}$ and $\beta_i=\gamma_i b_i \gamma_i^{-1}$ -- note that these are closed curves based at $p_0$ (see Figure~\ref{fig:surface}). For simplicity we will also use the notations $\alpha_i$ and $\beta_i$ for the corresponding classes in the fundamental group $\pi_1(\Sigma,p_0)$.

Let $(\partial \Sigma)_{p_0}$ denote the boundary of $\Sigma$ considered as a closed path from $p_0$ to itself oriented clockwise.
It is easy to check that $(\partial \Sigma)_{p_0}$ is homotopic to $\prod\limits_{i=1}^n [\alpha_i,\beta_i]$.

\vskip .2cm
The {\it mapping class group} $\Mod_n^1=\Mod(\Sigma)$ is defined as the subgroup of orientation preserving homeomorphisms of $\Sigma$ which fix the boundary $\partial \Sigma $ pointwise modulo the isotopies which fix $\partial \Sigma$ pointwise. 
The action of $\Mod(\Sigma)$ on $\Sigma$ induces an action on $\pi=\pi_1(\Sigma, p_0)$ and thus we obtain a homomorphism $\iota:\Mod(\Sigma)\to\Aut(\pi)$. The group $\pi$ is free of rank $2n$ with generators $\alpha_1,\beta_1,\ldots, \alpha_n,\beta_n$,
and since $(\partial \Sigma)_{p_0}$ is fixed under the action, the image of $\iota$ stabilizes $\prod\limits_{i=1}^n [\alpha_i,\beta_i]$ by property (ii) above. In fact, a stronger statement holds:

\begin{Theorem}
\label{thm:Zieschang}
The map $\iota$ is injective and $\Im\iota$ is equal to the full stabilizer of $\prod\limits_{i=1}^n [\alpha_i,\beta_i]$ in $\Aut(\pi)$.
\end{Theorem}

Theorem~\ref{thm:Zieschang} is proved, e.g., in \cite{ZiVC}: the above map $\iota$ is injective by \cite[Theorem~5.13.2]{ZiVC} and
surjective (that is, $\Im\iota$ is the full stabilizer) by \cite[Theorem~5.7.1]{ZiVC}. 
The surjectivity part is originally due to Zieschang~\cite{Zi}.

\begin{Remark}\rm
(a) The surjectivity part of Theorem~\ref{thm:Zieschang} can be rephrased by saying that
 $\Mod(\Sigma)$ acts transitively on the set of (ordered) bases $\delta_1,\delta_1',\ldots,\delta_n,\delta_n'$
 of $\pi$ satisfying $\prod\limits_{i=1}^n [\delta_i,\delta'_i]=(\partial \Sigma)_{p_0}$. Recall that such bases 
were called {\it natural} in the introduction.

(b) Theorem~\ref{thm:Zieschang} is a variation of the classical Dehn-Nielsen-Baer theorem which asserts that
for a closed surface $\Sigma$ of genus $n$, the mapping class group $\Mod(\Sigma)$ is isomorphic to an index $2$
subgroup of the outer automorphism group of a surface group on $2n$ generators.
\end{Remark}
\vskip .12cm

{\bf Johnson filtration.} For each $k\in\dbN$ define $\calI_{n}^1(k)$ to be the kernel of the induced map $\Mod_n^1\to\Aut(\pi/\gamma_{k+1}\pi)$.
The filtration $\{\calI_{n}^1(k)\}_{k=1}^{\infty}$ is called the {\it Johnson filtration}. The first term of this filtration
$\calI_{n}^1=\calI_{n}^1(1)$ is the {\it Torelli subgroup} of $\Mod_n^1$. It can also be defined as the set of elements of $\Mod_{n}^1$ acting trivially on $H_1(\Sigma_n^1)$. The second term of the Johnson filtration $\calK_{n}^1=\calI_{n}^1(2)$ is known as the {\it Johnson kernel}. 
One can characterize $\calK_{n}^1$ purely topologically as the subgroup generated by Dehn twists about separating curves.
The equivalence of these two definitions of $\calK_{n}^1$ is a deep theorem of Johnson~\cite{Jo:ab2}. 

Recall that we consider $n$ as being fixed, and for the rest of the section we will use the simplified notations 
$\calM=\Mod_n^1$, $\calI=\calI_n^1$ and $\calK=\calK_n^1$. Occasionally we will also use the notations $\calI(\Omega)$ and
$\calK(\Omega)$ for the Torelli subgroup (resp. Johnson kernel) of the mapping class group of a surface $\Omega$.

\subsection{Generators for the mapping class group} It is well known that the mapping class group
$\Mod_n^1$ is generated by Dehn twists. The minimal number of Dehn twists needed to generate $\Mod_n^b$ for $b=0,1$
is $2n+1$ as proved by Humphries~\cite{Hu} for $b=0$ and by Johnson~\cite{Jo:fg} for $b=1$ (the generating set in \cite{Jo:fg} is a natural analogue of the one in \cite{Hu}). More specifically, $\Mod_n^1$ is generated by the Dehn twists about the curves $c_1,c_2,\ldots, c_{2n},b$ defined in \cite[p.428, Figure~5]{Jo:fg}.

Usually, in the definition of the Dehn twist $T_{\gamma}$ one assumes that $\gamma$ is an essential simple closed curve, but for the discussion below it will be convenient to introduce the following convention: If $\gamma$ is a closed curve on $\Sigma$ which is not simple, but 
freely homotopic to some essential simple closed curve $\gamma'$, we set $T_{\gamma}=T_{\gamma'}$. The right-hand side is well defined since two freely homotopic curves on a surface are isotopic, and the Dehn twist $T_{\gamma'}$ is determined by the isotopy class of $\gamma'$.

With this convention, we can relate the Humphries-Johnson generating set to the curves $\alpha_i,\beta_i$ introduced in \S~\ref{sec:prelim}. It is not hard to see that $c_{2i}$ is freely homotopic to $\beta_i$ for $1\leq i\leq n$,  that $c_{2i-1}$ is freely homotopic to $\alpha_i\alpha_{i-1}^{-1}$ for $2\leq i\leq n$ and that $c_1$ and $b$ are freely homotopic to $\alpha_1$ and $\alpha_2$, respectively. Thus, \cite[Theorem~3]{Jo:fg} can be restated as follows:

\begin{Theorem} 
\label{thm:HuJo}
The mapping class group $\Mod_n^1$ is generated by the following Dehn twists: $\{T_{\beta_i}: 1\leq i\leq n\}$, $\{T_{\alpha_i\alpha_{i-1}^{-1}}: 2\leq i\leq n\}$, $T_{\alpha_1}$ and $T_{\alpha_2}$. 
\end{Theorem}

We will not explicitly refer to Theorem~\ref{thm:HuJo} in this paper, but we will use several results whose proof relies on 
Theorem~\ref{thm:HuJo}.

\subsection{Generators for the Torelli subgroup and subsurfaces $\Sigma_I$} It is a celebrated theorem of Johnson~\cite{Jo:fg} that the Torelli group $\calI=\calI_n^1$ is finitely generated for $n\geq 3$. Johnson's generating set from \cite{Jo:fg} is explicit, but it lacks a key feature of Magnus' generating set for $\IA_n$ -- the fact that most generating pairs commute -- that is essential for our purposes. A generating set for $\calI$ which has the latter property was constructed by Church and Putman~\cite{CP} using an earlier work of Putman~\cite{Pu1}.

\vskip .12cm
Recall the curves $\alpha_i$ and $\beta_i$, $1\leq i\leq n$, defined 
in \S~\ref{sec:prelim}. 
For each $I\subseteq [n]$ choose a subsurface $\Sigma_I\subseteq \Sigma$ satisfying the following properties:
\begin{itemize} 
\item[(i)] The curves $\alpha_i$ and $\beta_i$ lie on $\Sigma_I$ for all $i\in I$.
\item[(ii)] $\Sigma_I$ has genus $|I|$ and $1$ boundary component.
\item[(iii)] The boundary of $\Sigma_I$ is homotopic to $\prod_{i\in I}[\alpha_i,\beta_i]$ (where the product is taken in increasing order). 
\item[(iv)] $\Sigma_I\cap\partial\Sigma$ is an interval which contains $p_0$ and does not depend on $I$.
\end{itemize}
For a longer but more transparent definition of $\Sigma_I$ see \cite[\S~4]{CP} or \cite[\S~7]{EH}. 
For an illustration see Figure~\ref{figure:subsurface}.

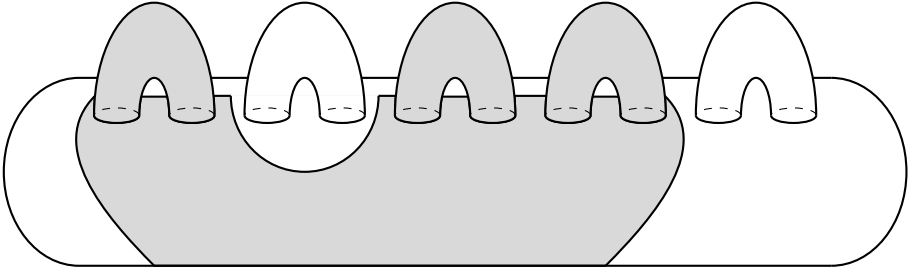
\begin{figure}
\begin{tikzpicture}

  %surface boundary parameters
  \def\xstart{0 }; %start and end of straight line part of surface
  \def\xend{10 };
  \def\xmid{0.5*\xstart + 0.5*\xend };
  \def\yb{2.5 }; %height of surface boundary
 
  %handle size parameters

  \def\y{2 }; %y-coord of handle hole centers
  \def\xr{0.3 }; %radii of handle holes
  \def\yr{0.1 };
  
  \def\hi{0.5 }; %vertical radius of inner part of handle
  \def\ho{1.5 }; %vertical radius of outer part of handle
  
  %handle position parameters
   %first handle
   \def\xa{0.5 }; %center of left hole of first handle
   \def\xb{1.5 }; %center of right hole of first handle
  
   %second handle
   \def\xc{2.5 };
   \def\xd{3.5 }; 
   
   %third handle
   \def\xe{4.5 };
   \def\xf{5.5 };
  
   %fourth handle
   \def\xg{6.5 };
   \def\xh{7.5 };
  
   %last handle
   \def\xi{8.5 };
   \def\xj{9.5 };
   
   %draws sigma134 boundary
   
   %\draw[thick, fill=gray!30] (\xstart,\yb) to[out=270, in=180] (4, 0.5) to[out=0, in=270] (8, \yb);

  %points for sigma134 boundary on handle arcs
  \path[] (\xb + \xr, \y) arc (0:10: 0.5*\xb-0.5*\xa + \xr and \ho) coordinate[] (A); %1st handle outer right
  
  \path[] (\xb - \xr, \y) arc (0:30: 0.5*\xb-0.5*\xa-\xr and \hi) coordinate[] (AR); %1st handle inner right
  \path[] (\xb - \xr, \y) arc (0:150: 0.5*\xb-0.5*\xa-\xr and \hi) coordinate[] (AL) ;%1st handle inner left
  
  \path[] (\xh + \xr, \y) arc (0:10: 0.5*\xh-0.5*\xg + \xr and \ho) coordinate[] (B); %4th handle outer right
  \path[] (\xh - \xr, \y) arc (0:150: 0.5*\xh-0.5*\xg-\xr and \hi) coordinate[] (CL); %4th handle inner left
  \path[] (\xh - \xr, \y) arc (0:30: 0.5*\xh-0.5*\xg-\xr and \hi) coordinate[] (CR); %4th handle inner right
  
  \path[] (\xf + \xr, \y) arc (0:10: 0.5*\xf-0.5*\xe + \xr and \ho) coordinate[] (D); %3rd handle outer right
  
  \path[] (\xf - \xr, \y) arc (0:150: 0.5*\xf-0.5*\xe-\xr and \hi) coordinate[] (EL); %3rd handle inner left
  \path[] (\xf - \xr, \y) arc (0:30: 0.5*\xf-0.5*\xe-\xr and \hi) coordinate[] (ER); %3rd handle inner right
  
  \path[] (3, 1.25) to[out=0, in=270] coordinate[pos=0.872] (WR) (4, \yb);
  \path[] (2, \yb) to[out=270, in=180] coordinate[pos=0.128] (WL) (3, 1.25);

   \draw[thick, fill=gray!30] (WL) -- (A) arc (10:0: 0.5*\xb-0.5*\xa + \xr and \ho) arc (0:-180: \xr and \yr) arc (0:30: 0.5*\xb-0.5*\xa-\xr and \hi) -- (AL) arc (150:180: 0.5*\xb-0.5*\xa-\xr and \hi) arc (0:-180: \xr and \yr) arc (180:170: 0.5*\xb-0.5*\xa + \xr and \ho) to[out=225, in=135] (1,0) -- (7,0) to[out=45, in=-45] (B) arc (10:0: 0.5*\xh-0.5*\xg + \xr and \ho) arc (0:-180: \xr and \yr) arc (0:30: 0.5*\xh-0.5*\xg-\xr and \hi) -- (CL) arc (150:180: 0.5*\xh-0.5*\xg-\xr and \hi) arc (0:-180: \xr and \yr) arc (180:170: 0.5*\xh-0.5*\xg + \xr and \ho) -- (D) arc (10:0: 0.5*\xf-0.5*\xe + \xr and \ho) arc (0:-180: \xr and \yr) arc (0:30: 0.5*\xf-0.5*\xe-\xr and \hi) -- (EL) arc (150:180: 0.5*\xf-0.5*\xe-\xr and \hi) arc  (0:-180: \xr and \yr)  arc (180:170: 0.5*\xf-0.5*\xe + \xr and \ho) -- (WR);

  \draw[thick, fill=white] (WL) to[out=270, in=180] (3, 1.25) to[out=0, in=270] (WR);
  
  %draws surface boundary
  \draw[thick] (\xstart,0) -- (\xend,0);
  \draw[thick] (0,\yb) -- (0.25, \yb) (1.75,\yb) -- (2.25, \yb) (3.75,\yb) -- (4.25,\yb) (5.75,\yb) -- (6.25,\yb) (7.75,\yb) -- (8.25,\yb) (9.75, \yb) -- (10, \yb);
  \draw[thick] (\xstart,\yb) arc (90:270:1 and 0.5*\yb);
  \draw[thick] (\xend,\yb) arc (90:-90:1 and 0.5*\yb);

  %First Handle

  \draw[thick, fill=gray!30] (\xb+\xr, \y) arc (0:-180: \xr and \yr) arc (0:180: 0.5*\xb-0.5*\xa-\xr and \hi) arc (0:-180: \xr and \yr) arc (180:0: 0.5*\xb-0.5*\xa + \xr and \ho);

  %draws left hole of handle
  \draw[dashed] (\xa+\xr, \y) arc (0:180: \xr and \yr);
  
  %draws right hole of handle
  \draw[dashed] (\xb+\xr, \y) arc (0:180: \xr and \yr);

  %Second Handle
  
  %draws left hole of handle
  \draw[dashed] (\xc+\xr, \y) arc (0:180: \xr and \yr);
  \draw[thick] (\xc+\xr, \y) arc (0:-180: \xr and \yr);
  
  %draws right hole of handle
  \draw[dashed] (\xd+\xr, \y) arc (0:180: \xr and \yr);
  \draw[thick] (\xd+\xr, \y) arc (0:-180: \xr and \yr);
  
  %draws handle part of handle
  \draw[thick] (\xd - \xr, \y) arc (0:180: 0.5*\xd-0.5*\xc-\xr and \hi);
  \draw[thick] (\xd + \xr, \y) arc (0:180: 0.5*\xd-0.5*\xc + \xr and \ho);

 %Third Handle
  
  \draw[thick, fill=gray!30] (\xf+\xr, \y) arc (0:-180: \xr and \yr) arc (0:180: 0.5*\xf-0.5*\xe-\xr and \hi) arc (0:-180: \xr and \yr) arc (180:0: 0.5*\xf-0.5*\xe + \xr and \ho);

  %draws left hole of handle
  \draw[dashed] (\xe+\xr, \y) arc (0:180: \xr and \yr);
  
  %draws right hole of handle
  \draw[dashed] (\xf+\xr, \y) arc (0:180: \xr and \yr);
 
  %Fourth Handle

\draw[thick, fill=gray!30] (\xh+\xr, \y) arc (0:-180: \xr and \yr) arc (0:180: 0.5*\xh-0.5*\xg-\xr and \hi) arc (0:-180: \xr and \yr) arc (180:0: 0.5*\xh-0.5*\xg + \xr and \ho);

  %draws left hole of handle
  \draw[dashed] (\xg+\xr, \y) arc (0:180: \xr and \yr);
  
  %draws right hole of handle
  \draw[dashed] (\xh+\xr, \y) arc (0:180: \xr and \yr);
  
  %Last Handle
  
  %draws left hole of handle
  \draw[dashed] (\xi+\xr, \y) arc (0:180: \xr and \yr);
  \draw[thick] (\xi+\xr, \y) arc (0:-180: \xr and \yr);
  
  %draws right hole of handle
  \draw[dashed] (\xj+\xr, \y) arc (0:180: \xr and \yr);
  \draw[thick] (\xj+\xr, \y) arc (0:-180: \xr and \yr);
  
  %draws handle part of handle
  \draw[thick] (\xj - \xr, \y) arc (0:180: 0.5*\xj-0.5*\xi-\xr and \hi);
  \draw[thick] (\xj + \xr, \y) arc (0:180: 0.5*\xj-0.5*\xi + \xr and \ho);
  
\end{tikzpicture}
\caption{The surface $\Sigma_{\{1,3,4\}}$}
\label{figure:subsurface}
\end{figure}

The subsurfaces $\Sigma_I$ are uniquely defined up to isotopy and satisfy the following properties (1)-(4). Properties (1)-(3)
follow immediately from the definitions and (4) can be proved by a standard application of the 
{\it change of coordinates principle}~\cite[1.1.3]{FM}.

\begin{Observation}
\label{obs:subsurfaces}
The following hold:
\begin{itemize}
\item[(1)] $\Sigma_{[n]}$ is isotopic to $\Sigma$.
\item[(2)] If $I\subseteq J$, there exist subsurfaces $\Sigma'_I$ isotopic to $\Sigma_I$
and $\Sigma'_J$ isotopic to $\Sigma_J$ such that $\Sigma'_I\subseteq \Sigma'_J$.
\item[(3)] If $I$ and $J$ are disjoint and uncrossed (as defined below), 
there exist subsurfaces $\Sigma'_I$ isotopic to $\Sigma_I$
and $\Sigma'_J$ isotopic to $\Sigma_J$ such that $\Sigma'_I$ and $\Sigma'_J$ are disjoint.
\item[(4)] If $|I|=|J|$, there exists an orientation-preserving homeomorphism $g$ of $\Sigma$
acting trivially on $\partial\Sigma$ such that $g(\Sigma_I)=\Sigma_J$.
\end{itemize}
\end{Observation}
\begin{Definition}\rm Let $I$ and $J$ be disjoint subsets of $[n]$. We will say that $I$ and $J$ are {\it crossed} if there exist $i_1,i_2\in I$ and $j_1,j_2\in J$
such that $i_1<j_1<i_2<j_2$ or $j_1<i_1<j_2<i_2$. Otherwise $I$ and $J$ will be called {\it uncrossed}. Clearly, if $I$ consists of consecutive integers, then $I$ is uncrossed with any subset $J$ disjoint from it.
\end{Definition}
\begin{Remark}\rm The technical condition (iv) in the definition of $\Sigma_I$ (which was not imposed in \cite{CP} or \cite{EH})
is needed to ensure that property (4) in Observation~\ref{obs:subsurfaces} holds. Note that an easier way to achieve (4) would be to require that $\Sigma_I\cap\partial\Sigma=\emptyset$. However, the latter would prevent us from considering the fundamental groups
of $\Sigma_I$ as subgroups of $\pi_1(\Sigma,p_0)$, something that is essential for our purposes.
\end{Remark}

For each $I\subseteq [n]$ define $\Mod_I$ to be the subgroup of $\Mod(\Sigma)$ consisting of mapping classes which have a representative supported on $\Sigma_I$, and let $\calI_I=\Mod_I\cap \calI$. Parts (1)-(3) 
of Observation~\ref{obs:subsurfaces} have obvious group-theoretic consequences:

\begin{Observation}The following hold:
\label{obs:uncrossed}
\begin{itemize}
\item[(i)] $\Mod_{[n]}=\Mod(\Sigma)$.
\item[(ii)] If $I\subseteq J$, then $\Mod_I\subseteq \Mod_J$.
\item[(iii)] If $I$ and $J$ are disjoint and uncrossed, then $\Mod_I$ and $\Mod_J$ commute.
\end{itemize}
\end{Observation}

A deep result of Church and Putman~\cite[Proposition~4.5]{CP} asserts
that if $n\geq 3$, then $\calI=\la \calI_I: |I|=3\ra$. This fact played a key role in the proof of finite generation of $\calK_n^1$ in \cite{EH}, but in \cite{EH} there was no need to work with a specific finite generating set inside $\bigcup\limits_{|I|=3}\calI_I$. In this paper we will need to choose such a generating set $S$ as follows: we will start with some finite generating set $S_{\{1,2,3\}}$ of $\calI_{\{1,2,3\}}$ and then add to it the images of $S_{\{1,2,3\}}$ under carefully chosen isomorphisms between $\calI_{\{1,2,3\}}$ and $\calI_I$ for every $3$-element subset $I$ of $[n]$. The details of this construction will be given later in this section.

\subsection{Abelian quotients of the Torelli subgroup}
Let $V=H_1(\Sigma)$. Then $V$ is a free abelian group of rank $2n$, and it is well known that the algebraic intersection form on $V$ is symplectic. Clearly $\calM/\calI$ acts on $V$ preserving this form, so there is a canonical group homomorphism
$\calM/\calI\to \Sp(V)$ where $\Sp(V)$ is the group of automorphisms of $V$ preserving this form. It is also well known that this homomorphism is an isomorphism, which enables us to identify $\calM/\calI$ with $\Sp(V)$. 

From the definition of the Johnson filtration it is easy to see that the quotient $\calI/\calK=\calI_{n}^1(1)/\calI_{n}^1(2)$ is abelian and torsion-free. A complete description of the abelianization of $\calI$ for $n\geq 3$ was obtained in a series of 
Johnson's papers~\cite{Jo:ab1,Jo:ab2,Jo:ab3}. Below we collect some specific results about abelian quotients of $\calI$ that will be needed in this paper:

\begin{Theorem} 
\label{thm:Johnson}
Assume that $n\geq 3$. The following hold:
\begin{itemize}
\item[(a)] $\calI/\calK$ is the largest torsion-free abelian quotient of $\calI$. 
\item[(b)] There is a canonical isomorphism of $\Sp(V)$-modules $\calI/\calK\cong \wedge^3 V$. In particular, $\calI/\calK\cong \dbZ^{{2n\choose 3}}$ as a group.
\item[(c)] $\calK/[\calI,\calI]$ (which by (a) is the torsion part of $\calI/[\calI,\calI]$) has exponent $2$ and rank
${2n\choose 2}+{2n\choose 1}+{2n\choose 0}$.
\end{itemize}
\end{Theorem}
\begin{Remark}\rm We briefly comment on how (a), (b) and (c) follow from the results of \cite{Jo:ab1,Jo:ab2,Jo:ab3}. In \cite{Jo:ab1},
Johnnson constructed
\begin{itemize}
\item an epimorphism $\tau: \calI\to \wedge^3 V$ such that $\calK\subseteq\Ker\tau$ and the induced map
$\calI/\calK\to \wedge^3 V$ is a homomorphism of $\Sp(V)$-modules and
\item a group epimorphism $\sigma:\calI\to B_3$ where $B_3$ is an elementary abelian $2$-group of rank $\sum\limits_{i=0}^3 {2n\choose i}$ (there is also a natural $\Sp(V)$-module structure on $B_3$, but it is not essential for our purposes). 
\end{itemize}
In \cite{Jo:ab2} it was proved that $\calK=\Ker\tau$ which yields (b). 
One of the main results of \cite{Jo:ab3} is that $[\calI,\calI]=\calK\cap \Ker\sigma$. This implies that
$\calI/[\calI,\calI]$ embeds into $\calI/\calK\oplus B_3$ which yields (a). The other main result of \cite{Jo:ab3} is that $B_3$ is the largest exponent $2$ quotient of $\calI$. Combined with (a) and (b) this implies (c).
\end{Remark}

We proceed with the description of $\calI/\calK$. Recall the curves $a_i$ and $b_i$ on $\Sigma$
introduced in \S~\ref{sec:prelim}. By slight abuse of notation below we will use the notations $a_i$ and $b_i$
for the corresponding homology classes in $H_1(\Sigma)$. It is clear that $a_1,b_1,\ldots,a_n,b_n$ is a symplectic basis of $H_1(\Sigma)$, that is, $a_i\cdot a_j=b_i\cdot b_j=0$ for all $i,j$ and $a_i\cdot b_j=\delta_{ij}$.
Also note that $a_i$ and $b_i$ span $H_1(\Sigma_{\{i\}})$. 
\medskip

Let $V=H_1(\Sigma)$, for $i\in [n]$ let $V_i=H_1(\Sigma_{\{i\}})=Span\{a_i,b_i\}$, and for $I\subseteq [n]$ let $V_I=\oplus_{i\in I}V_i=Span\{a_i,b_i: i\in I\}\subseteq V$. Thus, $\wedge^3 V=\sum_{|I|=3}\wedge^3 V_I$. Recall that we have an isomorphism of $Sp(V)$-modules
$\calI/\calK\to \wedge^3 V$. For every $I\subseteq [n]$ let $\phi_I: \calI_I\to \wedge^3 V$ be the map
obtained by precomposing the isomorphism $\calI/\calK\to \wedge^3 V$ with the natural projection 
$\calI\to\calI/\calK$ and the inclusion $\calI_I\to\calI$.

The following result is an immediate consequence of Johnson's paper~\cite{Jo:ab3}:
\begin{Lemma}
\label{lem:compatibility} For every $I\subseteq [n]$ we have $\phi_I(\calI_I)=V_I$.
\end{Lemma}
\vskip .1cm

We can now construct a generating set for $\calI$ with certain nice properties. First order the chosen basis of $V$ as follows:
$a_1<b_1<a_2<\ldots<b_n$, and let $\calB$ be the set of all wedges $x\wedge y\wedge z$ with $x,y,z\in \{a_i,b_i\}$ and $x<y<z$. Clearly 
$\calB$ is  a basis for $\wedge^3 V$. 

\begin{Lemma}
\label{lem:reggenset} Suppose that $n\geq 3$. Then for every $I\subseteq [n]$ with $|I|=3$ there exists a generating set $S_I$ for
$\calI_I$ with $|S_I|=42$ whose image in $\calI/\calK$ is equal to $(\calB\cap V_I)\sqcup\{0\}$. Moreover, if we let $S=\cup_{|I|=3} S_I$, then $S$
is generating set for $\calI$ which satisfies the conclusion of Lemma~\ref{lem:Nielsen} for $G=\calI$, $K=\calK$ and 
$E=\calB$.
\end{Lemma}
\begin{proof} 

By~\cite{Jo:fg}, the group $\calI_I\cong \calI_3^1$ has a generating set $S_I^0$ with $42$ elements. 
By Lemma~\ref{lem:compatibility}, the image of $\calI_I$ in $V$ is equal to $V_{I}$, which is a free abelian group of rank $20$ with basis $\calB\cap V_I$. As in the proof of Lemma~\ref{lem:Nielsen}, applying suitable Nielsen moves to $S_I^0$ (replacing $x$ by $xy^{\pm 1}$ where $x$ and $y$ are distinct generators), we obtain another generating set $S_I$ of $\calI$ with $42$ elements whose image in $V$ is equal to $(\calB\cap V_I)\sqcup\{0\}$.

The set $S=\bigcup_{|I|=3}S_I$ generates $\calI$ by \cite[Proposition~4.5]{CP}. Also the image of $S$ in 
$\wedge^3 V$ is equal to
$\bigcup_{|I|=3}(\calB\cap V_I)\sqcup\{0\}=\calB\sqcup\{0\}$, so $S$ indeed satisfies the conclusion of Lemma~\ref{lem:Nielsen}
(for the desired $G,K$ and $E$). 
\end{proof}

Our effective generation procedure described in Section~2 can be applied to any generating set $S$
constructed in Lemma~\ref{lem:reggenset}. However, some additional compatibility assumptions on $\{S_I\}$ will be needed in order to explicitly estimate the constants $A$ and $B$ from Theorem~\ref{prop:connectedradius}. We postpone this discussion until the end of this section. 

\subsection{Some elements of $\Mod(\Sigma)$} In this subsection we will prove the existence 
of certain elements of $\Mod(\Sigma)$ with a prescribed induced action on $\pi_1(\Sigma,p_0)$. These results will play a key role in estimating the constant $A$ later this section.

\begin{Lemma}
\label{lem:iso} Let $I$ be a non-empty subset of $[n]$, and let $k=|I|$. Suppose we are given closed curves
$\{\delta_j,\delta'_j\}_{j=1}^k$ which are based at $p_0$, lie on $\Sigma_I$, form a basis
of $\pi_1(\Sigma_{I},{p_0})$ and such that
$\prod_{j=1}^k [\delta_i,\delta'_i]=(\partial \Sigma_{I})_{p_0}$ in $\pi_1(\Sigma,p_0)$.
Then there exists an element $f\in \Mod(\Sigma)$ such that
\begin{itemize}
\item[(a)] some representative of $f$ maps $\Sigma_{[k]}$ onto $\Sigma_I$;
\item[(b)] the induced action of $f$ on $\pi_1(\Sigma,p_0)$ satisfies $f^*(\alpha_j)=\delta_{j}$ and $f^*(\beta_j)=\delta'_{j}$ for all $1\leq j\leq k$.
\end{itemize}
Moreover, suppose that $f'\in \Mod(\Sigma)$ is another element satisfying (a) and (b). Then $f'$ coincides with $f$ on $\Sigma_{[k]}$, and therefore,
$f s f^{-1}=f' s (f')^{-1}$ for all $s\in \Mod_{[k]}$.
\end{Lemma}
\begin{proof} We first construct an element $f$ satisfying (a) and (b).
By Observation~\ref{obs:subsurfaces}(4), there exists $g\in \Mod(\Sigma)$ such that some representative of $g$ maps $\Sigma_{[k]}$ to $\Sigma_{I}$.

Since $(\partial \Sigma_{[k]})_{p_0}= \prod\limits_{j=1}^k [\alpha_j,\beta_j]$ and 
$(\partial \Sigma_{I})_{p_0}= \prod\limits_{j=1}^k [\delta_j,\delta'_j]$ in $\pi_1(\Sigma,p_0)$, 
the induced action of $g$ on $\pi_1(\Sigma,p_0)$ satisfies
\begin{equation}
\label{eq:gI} 
g^*(\prod\limits_{j=1}^k [\alpha_j,\beta_j])=\prod\limits_{j=1}^k [\delta_j,\delta'_j].
\end{equation}

Since a representative of $g$ maps $\Sigma_{[k]}$ homeomorphically onto $\Sigma_I$ and
$\{\alpha_j,\beta_j\}_{j=1}^k$ generate $\pi_1(\Sigma_{[k]},p_0)$, the elements 
$\{g^*(\alpha_j),g^*(\beta_j)\}_{j=1}^k$ generate $\pi_1(\Sigma_I,p_0)$. 

Since $\pi_1(\Sigma_I,p_0)$ is free of rank $2k$ and the automorphism group of a free group $F$ acts transitively on the sets
of bases of $F$, there exists 
$\phi\in\Aut\, (\pi_1(\Sigma_I, p_0))$ such that $\phi(\delta_j)=g^*(\alpha_j)$ and $\phi(\delta'_{j})=g^*(\beta_j)$ for all $j\in [k]$. 
Then $$\phi(\prod\limits_{j=1}^k [\delta_j,\delta'_j])=\prod\limits_{j=1}^k [\phi(\delta_j),\phi(\delta'_{j})]
=\prod\limits_{j=1}^k [g^*(\alpha_j),g^*(\beta_j)]=g^*(\prod\limits_{j=1}^k [\alpha_j,\beta_j])=
\prod\limits_{j=1}^k [\delta_j,\delta'_j],$$
(where the last equality holds by \eqref{eq:gI}), so
by Theorem~\ref{thm:Zieschang} applied to the surface $\Sigma_I$, there exists $h\in\Mod(\Sigma_I)$ such that
$h^*=\phi$. If we extend $h$ to $\Mod(\Sigma)$ by
letting it act trivially on $\Sigma\setminus \Sigma_I$, then clearly $f=h^{-1}g$ satisfies both (a) and (b). 
\vskip .12cm
\vskip .12cm
Let us now prove the `moreover' part. Suppose that $f'\in \Mod(\Sigma)$ satisfies both (a) and (b), and let $h=(f')^{-1}f$. By (a) some representative of $h$ stabilizes $\Sigma_{[k]}$. By (b) $h^*$ acts trivially on 
$\pi_1(\Sigma_{[k]})$, and hence by the injectivity part of Theorem~\ref{thm:Zieschang} $h$ acts trivially on $\Sigma_{[k]}$,
so $f'$ coincides with $f$ on $\Sigma_{[k]}$. Finally, any $s\in \Mod_{[k]}$ is supported on $\Sigma_{[k]}$, so $hs=sh$, which yields the last assertion. 
\end{proof}

The following corollary describes a key special case of Lemma~\ref{lem:iso}.

\begin{Corollary}
\label{cor:iso} Let $I$ be a subset of $[n]$, let $k=|I|$, and let $i_1<i_2<\ldots<i_k$ be the elements of $I$ listed in increasing order. 
Then there exists an element $f_I\in \Mod(\Sigma)$ such that
\begin{itemize}
\item[(a)] some representative of $f_I$ maps $\Sigma_{[k]}$ onto $\Sigma_I$;
\item[(b)] the induced action of $f_I$ on $\pi_1(\Sigma)$ satisfies $f_I^*(\alpha_j)=\alpha_{i_j}$ and $f_I^*(\beta_j)=\beta_{i_j}$ for all $1\leq j\leq k$.
\end{itemize}
Moreover, if $f_I\in \Mod(\Sigma)$ is any element satisfying (a) and (b), then
for any subset $J$ of $[k]$, some representative of $f_I$ maps $\Sigma_J$ onto $\Sigma_{I_J}$ where $I_J=\{i_j: j\in J\}$.
\end{Corollary}
\begin{proof} The first assertion is a special case of Lemma~\ref{lem:iso}, so we only need to prove the `moreover' part. So take any $f_I\in \Mod(\Sigma)$ satisfying (a) and (b). Then
\begin{equation}
\label{eq:sigma}
f_I^*((\partial \Sigma_{J})_{p_0})=f_I^*(\prod\limits_{j\in J}[\alpha_j,\beta_j])=\prod\limits_{i\in I_J}[\alpha_i,\beta_i]=
(\partial \Sigma_{I_J})_{p_0}.
\end{equation} 
Thus, if $\widetilde f_I$ is any orientation-preserving homeomorphism of $\Sigma$ which fixes 
$\partial \Sigma$ pointwise and represents $f_I$, then $\widetilde f_I(\partial \Sigma_{J})$ is homotopic (and 
hence isotopic) to $\partial \Sigma_{I_J}$. By \cite[Prop.~1.11]{FM}, an isotopy between 
$\widetilde f_I(\partial \Sigma_{J})$ and $\partial \Sigma_{I_J}$ can be extended to an isotopy of $\Sigma$ (acting trivially on $\partial\Sigma$).
Hence after perturbing $\widetilde f_I$ by an isotopy of $\Sigma$, we can assume that
$\widetilde f_I(\partial \Sigma_{J})=\partial \Sigma_{I_J}$ and hence
$\widetilde f_I(\Sigma\setminus \partial \Sigma_{J})=\Sigma\setminus \partial \Sigma_{I_J}$.

The topological spaces $\Sigma\setminus \partial \Sigma_{J}$ and $\Sigma\setminus \partial \Sigma_{I_J}$
both have two connected components: 
$\Sigma_J\setminus \partial\Sigma_J$ and $\Sigma\setminus\Sigma_J$ (resp. $\Sigma_{I_J}\setminus \partial\Sigma_{I_J}$ and $\Sigma\setminus\Sigma_{I_J}$). Since $\widetilde f_I$ acts trivially on
$\partial\Sigma$ and since the intersection 
$(\Sigma\setminus\Sigma_J)\cap (\Sigma\setminus\Sigma_{I_J})\cap \partial\Sigma$ is non-trivial by construction,
$\widetilde f_I$ must map $\Sigma\setminus\Sigma_J$ to $\Sigma\setminus\Sigma_{I_J}$.
Thus $\widetilde f_I(\Sigma_J\setminus \partial\Sigma_J)=\Sigma_{I_J}\setminus \partial\Sigma_{I_J}$
and hence $\widetilde f_I(\Sigma_J)=\Sigma_{I_J}$. 
\end{proof}

\subsection{An analogue of Lemma~\ref{lemma:SL_n}} In this subsection we will establish an analogue of Lemma~\ref{lemma:SL_n} for mapping class groups
(see Lemma~\ref{lemma:Sp_{2n}} below). The proof of the second part of Lemma~\ref{lemma:Sp_{2n}} will be postponed till the next subsection, as estimation of the constants $A$ and $B$ in the mapping class group case requires more work.

Since $\calI/\calK$ is the maximal torsion-free abelian quotient of $\calI$, we have $\Hom(\calI,\dbR)=\Hom(\calI/\calK,\dbR)$. 
By the same logic as in Section~3, there is a natural embedding of $\calM=\Mod(\Sigma)$ into $\Aut(\calI)$ and a natural
action of $\Sp(V)=\calM/\calI$ on $\Hom(\calI/\calK,\dbR)$.

The symplectic group $\Sp(V)\cong \Sp_{2n}(\dbZ)$ is generated by the elements $w_i$, $\tau_i$ for $i\in [n]$ and $\tau_{ij}$ for $i\neq j\in [n]$ defined as follows (all basis
elements whose image is not specified are fixed):
\[w_i\colon\begin{cases}a_i\mapsto b_{i}\\
b_i\mapsto -a_{i}\end{cases}
\qquad
\tau_{i}\colon a_i\mapsto a_i+ b_i\\
\qquad
\tau_{ij}\colon \begin{cases}a_i\mapsto a_i+ a_j\\
b_j\mapsto b_j-b_i\end{cases}.\]
In addition, for each $1\leq i\neq j\leq n$ let $f_{ij}\in \Sp(V)$ be the element which swaps $a_i$ and $a_j$ and swaps $b_i$ and $b_j$. We will refer to the elements $f_{ij}$ as transpositions.

 \begin{Lemma}
 \label{lemma:Sp_{2n}}
 Let $S$ be as in Lemma~\ref{lem:reggenset}, let $\chi$ be a nonzero character of $G=\calI$ and let $M=M(\chi)$. The following hold:
\begin{itemize}
\item[(a)] There exists $g\in \Sp_{2n}(\dbZ)$ and $Z\subset S$ with the following properties: 
\begin{itemize}
\item[(i)] $(S,Z)$ is chain-centralizing.
\item[(ii)] $|g\chi(z)|\geq M$ for all $z\in Z$.
\item[(iii)] $g$ is a product of at most 15 elements of the form $\tau_{ij}$ and 
and at most 3 transpositions $f_{ij}$.
\end{itemize}
\item[(b)] There exists an absolute constant $C_0$ (independent of $n$) such that any $g$ in part (a) 
admits a lift $\phi\in\Mod_n^1$ with $A(\phi^{\pm 1})\leq C_0$.
\end{itemize}
  \end{Lemma}
 \begin{Remark}\rm
Arguing exactly as in the proof of Corollary~\ref{cor:reghypIAn}, we deduce from Lemma~\ref{lemma:Sp_{2n}} that the pair
$(G,K)=(\calI,\calK)$ satisfies the Regularity Hypothesis for $C=1$ and some finite $\Phi$ with $B(\Phi)\leq A(\Phi)\leq C_0$. 
 \end{Remark}
  
 \begin{proof} Unlike the case of $\Aut(F_n)$, the set $Z$ will depend on $\chi$, but there will be only boundedly many possibilities, and in each case $Z$ will contain $4$ elements, exactly one element from each of the sets $S\cap G_{I}$ for $I=\{1,2\},\{3,4\},\{5,6\}$ and $\{7,8\}$.
For each such $Z$ the pair $(S,Z)$ is chain-centralizing -- this follows immediately from Observation~\ref{obs:uncrossed}(iii).
\medskip

The proof of Lemma~\ref{lemma:Sp_{2n}}(a) is very similar to that of Lemma~\ref{lemma:SL_n}(a), so we will just list the main steps and skip the details of the computations. As in the proof of Lemma~\ref{lemma:SL_n}, in each step $M$ denotes a positive real number and $\lambda$ is an arbitrary character of $G$.

\medskip 
\noindent
{\it Step 1:} If $M(\lambda)\geq M$, there exists $g\in \Sp_{2n}(\dbZ)$ which  is a product of at most 3 transpositions $f_{ij}$
such that $gz\in \calB\cap V_{\{1,2\}}$ or $gz\in \calB\cap V_{\{1,2,3\}}$.

Below we will consider the case where $z$ in Step~1 lies in $V_{\{1,2\}}$; the other case is similar. Without loss of generality we can assume that $z=a_1\wedge b_1\wedge b_2$.

\medskip
\noindent
 {\it Step 2:} If $|\lambda(a_1\wedge b_1\wedge b_2)|\geq M$, there exists $g=\tau_{13}^e$ with $|e|\leq 2$ such that
 $|g\lambda(a_1\wedge b_1\wedge b_2)|\geq M$ and $|g\lambda(a_1\wedge b_3\wedge b_2)|\geq M$.
 
 \medskip
\noindent
 {\it Step 3:} If $|\lambda(a_1\wedge b_1\wedge b_2)|\geq M$ and $|\lambda(a_1\wedge b_3\wedge b_2)|\geq M$, there exists
 $g=\tau_{31}^e$ with $|e|\leq 2$ such that
 $|g\lambda(a_1\wedge b_1\wedge b_2)|\geq M$ and $|g\lambda(a_3\wedge b_3\wedge b_2)|\geq M$.

\medskip 
\noindent
 {\it Step 4:} $|\lambda(a_1\wedge b_1\wedge b_2)|\geq M$ and $|\lambda(a_3\wedge b_3\wedge b_2)|\geq M$, there exists
 $g=\tau_{24}^{e}$ with $|e|\leq 1$ such that $|g\lambda(a_1\wedge b_1\wedge b_2)|\geq M$ and $|g\lambda(a_3\wedge b_3\wedge b_4)|\geq M$.
 
\medskip
 
 Combining Steps~1-4, we conclude that there exists $g_1\in\Sp_{2n}(\dbZ)$ equal to the product of at most 3 transpositions and at most 5 elements 
of the form $\tau_{13}^{\pm 1},\tau_{31}^{\pm 1}, \tau_{24}^{\pm 1}$ such that $|g_1\chi(z)|\geq M$ for some $z\in V_{\{1,2\}}\cap \calB$ and also for some 
$z\in V_{\{3,4\}}\cap \calB$. Repeating Steps~2-4 two more times, first replacing indices $3$ and $4$ by $5$ and $6$, respectively, and then by
$7$ and $8$, we obtain an element $g$ satisfying the conclusion of Lemma~\ref{lemma:Sp_{2n}}.
\medskip

(b) For an essential simple closed curve $\alpha$, the action of $T_{\alpha}$ on $H_1(\Sigma)$ is given by the following formula (see, e.g., \cite[Prop. 6.3]{FM}):
\footnote{Proposition~6.3 in \cite{FM} makes the additional assumption that the curve $\beta$ is also simple.
However, since $H_1(\Sigma)$ is generated by the homology classes of simple closed curves, by linearity 
\eqref{eq:DehnTwistAction} holds for all $\beta$.}
\begin{equation}
\label{eq:DehnTwistAction}
T_{\alpha}([\beta])=[\beta]+([\beta]\cdot [\alpha])[\alpha].
\end{equation}
A direct computation using this formula shows that
\begin{itemize}
\item[(i)] The Dehn twist $T_{\beta_i}$ is a lift of $\tau_i$;
\item[(ii)] the element $W_i=T_{\beta_i}T_{\alpha_i}^{-1} T_{\beta_i}$ is a lift of $w_i$;
\item[(iii)] the element $W_i T_{\alpha_i\alpha_j^{-1}}^{-1} T_{\alpha_i} T_{\alpha_j} W_i^{-1}$ is a lift 
of $\tau_{i,j}$.
\end{itemize}
Below we exhibit the calculation for (iii) (assuming the result for (ii)). Recall that $a_i=[\alpha_i]$
and $b_i=[\beta_i]$. We have
\begin{align*}
&a_i
\stackrel{W_i^{-1}}{\longmapsto}-b_i
\stackrel{T_{\alpha_j}}{\longmapsto} -b_i
\stackrel{T_{\alpha_i}}{\longmapsto} a_i-b_i
\stackrel{T_{\alpha_i\alpha_j^{-1}}^{-1}}{\longmapsto} a_i-b_i-(a_i-a_j)=a_j-b_i
\stackrel{W_i}{\longmapsto}a_j+a_i
&\\
&b_i
\stackrel{W_i^{-1}}{\longmapsto}a_i
\stackrel{T_{\alpha_j}}{\longmapsto} a_i
\stackrel{T_{\alpha_i}}{\longmapsto} a_i
\stackrel{T_{\alpha_i\alpha_j^{-1}}^{-1}}{\longmapsto} a_i
\stackrel{W_i}{\longmapsto}b_i
&\\
&a_j
\stackrel{W_i^{-1}}{\longmapsto}a_j
\stackrel{T_{\alpha_j}}{\longmapsto} a_j
\stackrel{T_{\alpha_i}}{\longmapsto} a_j
\stackrel{T_{\alpha_i\alpha_j^{-1}}^{-1}}{\longmapsto} a_j
\stackrel{W_i}{\longmapsto}a_j
&\\
&b_j
\stackrel{W_i^{-1}}{\longmapsto}b_j
\stackrel{T_{\alpha_j}}{\longmapsto} b_j-a_j
\stackrel{T_{\alpha_i}}{\longmapsto} b_j-a_j
\stackrel{T_{\alpha_i\alpha_j^{-1}}^{-1}}{\longmapsto} b_j-a_j-(a_i-a_j)=b_j-a_i
\stackrel{W_i}{\longmapsto}b_j-b_i.
\end{align*}
In the case of transpositions $f_{ij}$ it will be more convenient to define lifts directly instead of expressing them as products of Dehn twists. Given distinct $i<j\in [n]$, let $F_{ij}$ be the unique element of $\Mod(\Sigma)$ which is supported on $\Sigma_{\{i,j\}}$ and acts on $\pi_1(\Sigma_{\{i,j\}})$ as follows:
\begin{align*}
&\alpha_i\mapsto \alpha_j&
&\beta_i\mapsto \beta_j&
&\alpha_j\mapsto \alpha_i^{[\alpha_j,\beta_j]}&
&\beta_j\mapsto \beta_i^{[\alpha_j,\beta_j]}.&
\end{align*}
Such an element $F_{ij}$ exists (and is unique) by Lemma~\ref{lem:iso}. It is clear that $F_{ij}$ is a lift of 
$f_{ij}$.

Now any $g$ in part (a) has a lift $\phi$  which can be written as a product of at most $3$ transposition-lifts $F_{ij}$ and at most $135$ Dehn twists $T_{\alpha_i}$, $T_{\beta_i}$ or $T_{\alpha_i\alpha_j^{-1}}$
or their inverses. 
By \eqref{eq:AB}, in order to get an absolute bound for $A(\phi^{\pm 1},S)$, it suffices to prove the following proposition.

\begin{Proposition} 
\label{prop:Aestimate}
For a suitable choice of $S$ there exists an absolute constant $C$ (independent of $n$)
such that $A(g,S)\leq C$ for $g=T_{\alpha_i}^{\pm 1}$, $T_{\beta_i}^{\pm 1}$, $T_{\alpha_i \alpha_j^{-1}}^{\pm 1}$ or $F_{ij}^{\pm 1}$.
\end{Proposition}
\begin{Remark}\rm The assertion of Proposition~\ref{prop:Aestimate} does not appear to be obvious even if we restrict ourselves to, say,
$g=T_{\alpha_i}$ with $i$ fixed but $n$ tending to infinity.
\end{Remark}

Proposition~\ref{prop:Aestimate} will be proved in the next subsection.
\end{proof}

\subsection{Estimating the $A$ constants}

Recall that by Corollary~\ref{cor:iso}, for every non-empty subset $I=\{i_1<i_2<\ldots< i_k\}$ of $[n]$
there exists an element $f_I\in\Mod(\Sigma)$ which maps $\Sigma_{[k]}$ to $\Sigma_I$ and satisfies
$f_I^*(\alpha_j)=\alpha_{i_j}$ and $f_I^*(\beta_j)=\beta_{i_j}$ for all $1\leq j\leq k$. 
From now on we will fix such an element $f_I$ for every $I$.

Let us record one more simple observation, which is an immediate consequence of the moreover parts of
Lemma~\ref{lem:iso} and Corollary~\ref{cor:iso}.

\begin{Observation}
\label{obs:iso}
Let $I=\{i_1<i_2<\ldots< i_k\}$ be a subset of $[n]$, let $J$ be a subset of $[k]$, and let $I_J=\{i_j: j\in J\}$. Then
the elements $f_{I_J}$ and $f_I f_J$ coincide on $\Sigma_{[|J|]}$ and hence $f_{I_J}s f_{I_J}^{-1}=f_I f_J s (f_I f_J)^{-1}$
for all $s\in \Mod_{[|J|]}$.
\end{Observation}
\begin{proof}
Let $t=|J|$, and write $J=\{j_1<\ldots<j_t\}$.
Recall that some representative of $f_J$ maps $\Sigma_{[k]}$ to $\Sigma_J$, and by 
Corollary~\ref{cor:iso} some representative of $f_I$ maps $\Sigma_J$ to $\Sigma_{I_J}$. Thus, if
$f=f_I f_J$, the following hold:
\begin{itemize}
\item[(i)] Some representative of $f$ maps $\Sigma_{[k]}$ to $\Sigma_{I_J}$.
\item[(ii)] For all $1\leq m\leq t$ we have $f^*(\alpha_m)=f_I^*(\alpha_{j_m})=\alpha_{i_{j_m}}$
and similarly $f^*(\beta_m)=\beta_{i_{j_m}}$.
\end{itemize}
If $f=f_{I_J}$, then both (i) and (ii) also hold by construction. Hence, the assertion of
Observation~\ref{obs:iso} follows from the moreover part of Lemma~\ref{lem:iso}.
\end{proof}

We are now ready to prove Proposition~\ref{prop:Aestimate}:

\begin{proof}[Proof of Proposition~\ref{prop:Aestimate}] Define a generating set $S$ for $\calI$ as follows:

\medskip
\noindent
{\bf Definition of $S$:} Choose any generating set $S_{[3]}$ of $\calI_{[3]}$ satisfying the requirement of Lemma~\ref{lem:reggenset}. For every $I\subseteq [n]$ with $|I|=3$ define $S_I=f_I S_{[3]} f_I^{-1}$ and let $S=\bigcup_{|I|=3}S_I$. It is clear that each $S_I$
(and hence $S$) also satisfy the requirement of Lemma~\ref{lem:reggenset}.
\medskip

By definition, $A(g,S)=A(g,\bigcup\limits_{|L|=3}S_L)=\max\limits_{|L|=3} A(g,S_L)$, so it is enough to bound $A(g,S_L)$ with $|L|=3$.
From now on we fix $L$ with $|L|=3$.
\vskip .12cm
{\it Case 1:} $g=T_{\alpha_u\alpha_{v}^{-1}}^{\pm 1}$ for some $u\neq v$. 

Since the Dehn twist does not depend on the orientation of the curve, without loss of generality we can assume that
$u<v$. Below we will consider the subcase $g=T_{\alpha_u\alpha_{v}^{-1}}$; the subcase $g=T_{\alpha_u\alpha_{v}^{-1}}^{-1}$ is analogous. 

Let $I=L\cup\{u,v\}$ (thus $3\leq |I|\leq 5$). Let
$i_1<i_2<\ldots<i_{|I|}$ be the elements of $I$ listed in increasing order. Then, in the notations of Observation~\ref{obs:iso}
we have $L=I_J$ where $J$ is a 3-element subset of $[|I|]\subseteq [5]$.

Now take any element $x\in S_{L}$. By definition, $x=f_L s f_L^{-1}$ where $s\in S_{[3]}$. Since $f_L=f_I f_J$ on $\Sigma_{[3]}$
by Observation~\ref{obs:iso} and $f_J s f_J^{-1}\in S_{J}$ (by definition of $S_J$), we have $x=f_I t f_I^{-1}$ for some $t\in S_J$.
Hence
\begin{equation}
\label{eq:conj}
gxg^{-1}=g (f_I t f_I^{-1})g^{-1}=f_I (f_I^{-1}gf_I)t(f_I^{-1}gf_I)^{-1}f_I^{-1}.
\end{equation}
By the conjugation formula for Dehn twists (see, e.g.,\cite[Fact~3.7]{FM}) we have 
$$f_I^{-1}gf_I=f_I^{-1}T_{\alpha_u \alpha_{v}^{-1}}f_I=T_{(f_I^{-1})^*(\alpha_u \alpha_{v}^{-1})}=
T_{(f_I^{-1})^*(\alpha_u) (f_I^{-1})^*(\alpha_{v})^{-1}}=T_{\alpha_r \alpha_{s}^{-1}}.$$
where $r,s\in [|I|]$ are the unique indices such that $i_r=u$ and $i_s=v$.

It follows that $f_I^{-1}gf_I$ is an element of $\Mod_{[|I|]}$, for which we have only boundedly many 
(in fact, at most ${5 \choose 2}=10$) possibilities.
Hence we also have boundedly many possibilities for $y=(f_I^{-1}gf_I)t(f_I^{-1}gf_I)^{-1}$ (since $t\in S_J$ and $J\subset [5]$
with $|J|=3$, we have at most ${5\choose 3}|S_{[3]}|=420$ possibilities for $t$ and hence at most 4200 possibilities for $y$). Note that
$y$ also lies in $\calI_{[|I|]}$ and thus we can write $y=t_1^{\pm 1}\ldots t_m^{\pm 1}$ where each $t_j\in S_{[|I|]}$ and $m\leq C_1$
for some absolute constant $C_1$. By \eqref{eq:conj} we have
$gxg^{-1}=f_I y f_I^{-1}=\prod\limits_{k=1}^m f_I t_k f_I^{-1}$. We claim that each factor in the latter product lies in $S$.
This would imply that $\|gxg^{-1}\|_{S}\leq m\leq C$ and thus finish the proof.

Indeed, for each $k$ as above we have $t_k\in S_{J_k}$ with $|J_k|=3$ and $J_k\subseteq [|I|]$, so $t_k=f_{J_k} s_k f_{J_k}^{-1}$ for some $s_k\in S_{[3]}$. Hence, using Observation~\ref{obs:iso} again we have
$$f_I t_k f_I^{-1}=f_I f_{J_k} s_k f_{J_k}^{-1} f_I^{-1}=f_{I_{J_k}}s_k f_{I_{J_k}}^{-1}\in S_{I_{J_k}}\subset S.$$
\vskip .12cm
{\it Case 2:} $g=T_{\alpha_u}^{\pm 1}$ or $T_{\beta_u}^{\pm 1}$ for some $u$. This case is similar to (and easier than) Case~1. 
\vskip .12cm

{\it Case 3:}  $g=F_{uv}^{\pm 1}$ for some $u<v$. The argument in this case is mostly similar to Case~1, so we will just outline the differences. First as in Case~1, without loss of generality we can assume that $g=F_{uv}$.

Fix $L\subseteq [n]$ , and let $I=L \cup\{u,v\}$ (thus $3\leq |I|\leq 5$). Take any $x\in S_L$. As in Case~1, $x=f_I t f_I^{-1}$ with $t\in S_{[|I|]}$, and we just need to find an appropriate expression for the conjugate $f_I^{-1}F_{uv}f_I$.

Let $i_1<i_2<\ldots<i_{|I|}$ be the elements of $I$ listed in increasing order, 
and let $r,s\in [|I|]$ be such that $u=i_r$ and $v=i_s$.
By definition of $f_I$, some representative of $f_I$ sends $\Sigma_{\{r,s\}}$ to $\Sigma_{\{u,v\}}$. Since $F_{uv}$ is trivial on the complement of
$\Sigma_{\{u,v\}}$, we conclude that $f_I^{-1}F_{uv}f_I$ is trivial on the complement of $\Sigma_{\{r,s\}}$. The action
on $\Sigma_{\{r,s\}}$ is determined by direct computation below. We have
\begin{gather}
\left(f_I^{-1}F_{uv}f_I\right)^*(\alpha_r)=\left(f_I^{-1}F_{uv}\right)^*(\alpha_u)=\left(f_I^{-1}\right)^*(\alpha_v)=\alpha_s \\
\left(f_I^{-1}F_{uv}f_I\right)^*(\alpha_s)=\left(f_I^{-1}F_{uv}\right)^*(\alpha_v)=\left(f_I^{-1}\right)^*(\alpha_u^{\alpha_v})=\alpha_r^{\alpha_s}.
\end{gather}
Similarly, $\left(f_I^{-1}F_{uv}f_I\right)^*(\beta_r)=\beta_s$ and $\left(f_I^{-1}F_{uv}f_I\right)^*(\beta_s)=\beta_r^{\beta_s}$.
Hence $f_I^{-1}gf_I=f_I^{-1}F_{uv}f_I=F_{rs}$. Thus, $f_I^{-1}gf_I$ is supported on $\Sigma_{[|I|]}$, and
we have boundedly many possibilities for $f_I^{-1}gf_I$, and we can finish the proof as in Case~1.
\end{proof}
\subsection{Conclusion of the proof of Theorem~\ref{thm:main_Torelli}}
\label{sec:pfTorelli2}

\begin{proof}[Proof of Theorem~\ref{thm:main_Torelli}]
Let us first prove Theorem~\ref{thm:main_Torelli}(1). Let $S$ be a generating set for $\calI$ constructed in 
Proposition~\ref{prop:Aestimate}. This set need not satisfy the conclusion of Theorem~\ref{thm:main_Torelli}(1). First we decompose $S=S_1\sqcup S_2\sqcup S_3$ as in the setup introduced before Theorem~\ref{preimage_effective}. Thus, $S_3=S\cap \mathcal K$, $S_1=\{s_1,\ldots,s_N\}$ is a subset of $S\setminus S_3$ such that the natural projection $\theta:\mathcal I\to \mathcal I/\mathcal K$ maps $S_1$ bijectively onto $\mathcal B$ (and $S_2=S\setminus (S_1\sqcup S_3)$). Also recall that for $s\in S_2$ we denote by $d(s)$ the unique integer such that $\theta(s)=\theta(s_{d(s)})$.
\vskip .12cm

Define $S^{(1)}=S_1$ and $S^{(4)}=S_3\cup \{s^{-1}s_{d(s)}: s\in S_2\}$. Then $\la S^{(1)}\cup S^{(4)}\ra=
\la S_1\sqcup S_2\sqcup S_3\ra=\calI$. Also note that $S^{(4)}$ lies in $\calK$ and the image of $S^{(4)}$
in $\calI^{ab}$ generates $\calK/[\calI,\calI]$ (the latter holds since the image of $S^{(1)}\cup S^{(4)}$ generates $\calI^{ab}$ and the images of elements of $S^{(1)}$ are linearly independent modulo $\calK$). Since $\calK/[\calI,\calI]$ is a vector space over $\dbF_2$, we can choose a subset $S^{(2)}$ of $S^{(4)}$ whose image
in $\calK/[\calI,\calI]$ is a basis of $\calK/[\calI,\calI]$. Finally, we can multiply each element of $S^{(4)}\setminus S^{(2)}$ by a product of elements of $S^{(2)}$ (on either side) so that the obtained element lies in $[\calI,\calI]$, and let $S^{(3)}$ be the set of all such elements. 

Since $S^{(1)}\sqcup S^{(2)}\sqcup S^{(3)}$ is obtained from $S$ by Nielsen transformations, it clearly generates $\calI$, and the additional properties asserted in Theorem~\ref{thm:main_Torelli}(1) hold by construction.

\medskip
Now let $\theta:\calI\to \calI/\calK$ be the natural projection. By Theorem~\ref{prop:connectedradius} and the remark after Lemma~\ref{lemma:Sp_{2n}}, the set $\theta^{-1}(B_{\infty}(R))$ is connected where $R$ is bounded by an absolute constant.
By Theorem~\ref{preimage_effective}, $\calK$ is generated by elements of the form 
\begin{itemize}
 \item[(a)] $[s_i,s_j]^{s_i^{a_i}s_{i+1}^{a_{i+1}}\ldots\,\, s_n^{a_n}}$ where $1\leq i<j\leq n$ and $|a_j|\leq R$ for all $j$;
 \item[(b)] $x^{s_1^{a_1}s_{2}^{a_{2}}\ldots s_n^{a_n}}$ where $x\in S^{(4)}$ and $|a_j|\leq R$ for all $j$.
\end{itemize}
This generating set is almost the same as the set in  Theorem~\ref{thm:main_Torelli}(2) -- the only difference is that the condition $x\in S^{(4)}$ above is replaced by $x\in S^{(2)}\sqcup S^{(3)}$.
But by construction the subgroup generated by $S^{(4)}$ is equal to the subgroup generated by $S^{(2)}$ and $S^{(3)}$, so the set described
in Theorem~\ref{thm:main_Torelli}(2) also generates $G$. This proves Theorem~\ref{thm:main_Torelli}(2).

Finally, Theorem~\ref{thm:main_Torelli}(3) follows from Theorem~\ref{thm:main_Torelli}(2) by a standard application of the Reidemeister-Schreier process and straightforward computations involving basic commutator identities. 
\end{proof}

\paragraph{\bf Estimating the constant $R$.} Let us now briefly address the problem of explicitly estimating the constant $R$ in
Theorem~\ref{thm:main_Torelli}. Let $S_{[3]}^{0}$ be the generating set of $\calI_{[3]}^1$ with $42$ elements constructed in Johnson's paper~\cite{Jo:fg}. We can algorithmically construct another generating set $S_{[3]}$ of $\calI_{[3]}^1$, also with 42 elements, satisfying the conclusion of Lemma~\ref{lem:reggenset}. Next for each $I\subseteq [5]$ with $|I|=3$ we choose $f_I$ satisfying the conclusion of Corollary~\ref{cor:iso} (it is easy to do this explicitly) and define $S_I=f_I S_{[3]} f_I^{-1}$. Also define $S_{[5]}=\bigcup\limits_{|I|=3, I\subset [5]}S_I$.

Now the key step is estimating the constant $C_1$ from the proof of Proposition~\ref{prop:Aestimate}. To do this, we need to explicitly express each conjugate $yty^{-1}$, where $t\in S_{[5]}$ and $y=T_{\alpha_i\alpha_j^{-1}}^{\pm 1}$, 
$T_{\alpha_i}^{\pm 1}$ or $T_{\beta_i}^{\pm 1}$, $1\leq i\neq j\leq 5$, as a product of elements of $S_{[5]}$. An algorithm for
obtaining such expressions is given in the Ph.D. thesis of Stylianakis~\cite{Sty}; however, some of the arguments in \cite{Sty}
rely on certain computations in \cite{Jo:fg}.

Finally, it is easy to express the constant $R$ in terms of $C_1$ following the steps of the proof of Theorem~\ref{thm:main_Torelli}(2).

\subsection{Proof of Theorem~\ref{thm:main_Torelli2}}

Throughout this subsection $S$ will denote the generating set for $\calI$ constructed in Proposition~\ref{prop:Aestimate}
(not the modified set from the proof of Theorem~\ref{thm:main_Torelli}). Let us recall some notations from the introduction. We let $\omega=\{\alpha_i,\beta_i\}_{i=1}^n$; also for $1\leq l\leq n$ we set $\omega_l=\{\alpha_i,\beta_i\}_{i=1}^l$. Given $m\in\dbN$
we denote by $T_{sc}(m)$ the set of all Dehn twists $T_{\gamma}$ where $\gamma\in \pi_1(\Sigma, p_0)$ can be
represented by a separating curve and $\|\gamma\|_{\omega}\leq m$.  

Theorem~\ref{thm:main_Torelli2} will be deduced from Theorem~\ref{thm:main_Torelli} (or rather its proof)
and the following technical lemma.

\begin{Lemma}
\label{lem:expansion}
 The following hold:
\begin{itemize}
\item[(a)] Let $w$ be any word in the free group on 2 generators. There exists a constant 
$C_w$ depending only on $w$ with the following property: for every $s,t\in S$ such that
$w(s,t)\in\calK$, the element $w(s,t)$ lies in $\la T_{sc}(C_w)\ra$, the subgroup generated by
$T_{sc}(C_w)$.
 \item[(b)] There exists an absolute constant $C_2$ such that
$\|s^*(\alpha)\|_{\omega}\leq C_2\|\alpha\|_{\omega}$ for all $\alpha\in \pi_1(\Sigma,p_0)$ and $s\in S^{\pm 1}$.
\end{itemize}
\end{Lemma}

Before proving Lemma~\ref{lem:expansion}, we establish a simple auxiliary result.

\begin{Claim} 
\label{claim:Church}
For every $I\subseteq [n]$ we have $\calK(\Sigma_I)=\calK\cap \Mod_I$.
\end{Claim}
\begin{proof} Let $\pi=\pi_1(\Sigma,p_0)$ and $\pi_I=\pi_1(\Sigma_I,p_0)$.
Using Theorem~\ref{thm:Zieschang}, we can identify $\Mod(\Sigma)$ and $\Mod_I=\Mod(\Sigma_I)$ with subgroups
of $\Aut(\pi)$ and $\Aut(\pi_I)$, respectively. Under this identification, $\calK(\Sigma_I)$ consists of 
all elements of $\Aut(\pi_I)$ which act trivially modulo $\gamma_3 \pi_I$ and lie in $\Mod_I$, while 
$\calK\cap \Mod_I$ consists of  all elements of $\Aut(\pi_I)$ which act trivially modulo 
$\gamma_3 \pi \cap \pi_I$ and lie in $\Mod_I$.

Thus, proving the equality $\calK(\Sigma_I)=\calK\cap \Mod_I$ reduces to showing that $\gamma_3 \pi_I=\gamma_3 \pi\cap \pi_I$. The latter holds since $\pi$ is free and $\pi_I$ is a free factor of $\pi$ and hence a retract of $\pi$.
\end{proof}
\begin{Remark}\rm Claim~\ref{claim:Church} is a special case of \cite[Theorem~4.6]{Ch}; however, we decided to give a direct proof since we are dealing with a much more specific situation compared to \cite{Ch}.
\end{Remark}

\begin{proof}[Proof of Lemma~\ref{lem:expansion}](a) Take any $s,t\in S$ and a word $w$ such that $w(s,t)\in \calK$.
By definition of $S$, there exist $I,J\subseteq [n]$ with $|I|=|J|=3$ and $u,v\in S_{[3]}$ such that
$s=f_I u f_I^{-1}$ and $t=f_J v f_J^{-1}$. Let $L=I\cup J$ and $l=|L|$ (thus, $l\leq 6$). Then,
using the notations of Observation~\ref{obs:iso}, $I=L_A$ and $J=L_B$
for some $A,B\subseteq [l]\subseteq [6]$. Hence by Observation~\ref{obs:iso},
$s=f_{L_A} u f_{L_A}^{-1}=f_L s'f_L^{-1}$  where $s'=f_A u f_A^{-1}$. Similarly 
$t=f_L t'f_L^{-1}$ where $t'=f_B u f_B^{-1}$.

\vskip .12cm

Note that $s',t'\in S_{[l]}$. Also note that $s'$ and $t'$ are obtained from $s$ and $t$ by conjugation by the same element of $\Mod(\Sigma)$. Since $w(s,t)\in\calK$ and $\calK$ is normal in $\Mod(\Sigma)$, it follows that
$w(s',t')\in \calK \cap \Mod_{[l]}=\calK(\Sigma_{[l]})$ where the last equality holds by Claim~\ref{claim:Church}.

Since $s',t'\in S_{[l]}$ and $l\leq 6$, there are boundedly many possibilities for the pair $(s',t')$. Let us enumerate all such pairs $(s_1',t_1'),\ldots, (s_l',t_l')$ satisfying the additional restriction
$w(s_j',t_j')\in\calK$ (and hence $w(s_j',t_j')\in\calK(\Sigma_{[l]})$). Since $\calK(\Sigma_{[l]})$ is generated
by Dehn twists $T_{\gamma}$ with $\gamma\in \pi_1(\Sigma_{[l]},p_0)$ separating, there exists a constant $C_w$ depending only on $w$ such that each element $w(s_j',t_j')$ (and in particular $w(s',t')$) lies in the subgroup generated by all $T_{\gamma}$ where $\gamma\in \pi_1(\Sigma_{[l]},p_0)$ is separating and $\|\gamma\|_{\omega_l}\leq C_w$. 

Thus we can write $w(s',t')=\prod_{i=1}^{k} T_{\gamma_i}^{\pm 1}$
where each $\gamma_i\in \pi_1(\Sigma_{[l]},p_0)$ is separating with $\|\gamma_i\|_{\omega_l}\leq C_w$. Then
\begin{equation}
\label{eq:cw}
w(s,t)=w(f_L s'f_L^{-1}, f_L t'f_L^{-1})=f_L w(s',t') f_L^{-1}=\prod_{i=1}^{k} f_L T_{\gamma_i}^{\pm 1} f_L^{-1}=
\prod_{i=1}^{k} T_{f_L^*(\gamma_i)}^{\pm 1}.
\end{equation}
Finally, by definition of $f_L$ we have $f_L^*(\omega_l)\subseteq \omega$ and therefore 
for any $\gamma_i$ in \eqref{eq:cw} we have $\|f_L^*(\gamma_i)\|_{\omega}\leq \|\gamma_i\|_{\omega_l}\leq C_w$,
and hence $w(s,t)$ lies in $T_{sc}(C_w)$ by \eqref{eq:cw}.

\vskip .12cm

(b) Take any $1\neq \alpha\in \pi_1(\Sigma,p_0)$, let $k=\|\alpha\|_{\omega}$,
and write $\alpha=\prod\limits_{j=1}^k \lambda_j^{\pm 1}$ with $\lambda_j\in \omega$.
Since $\|s^*(\alpha)\|_{\omega}=\|\prod\limits_{j=1}^k s^*(\lambda_j)^{\pm 1}\|_{\omega}\leq
\sum\limits_{j=1}^k \|s^*(\lambda_j)\|_{\omega}$, to prove (b) it suffices to show that
there exists an absolute constant $C_2$ such that
$$\|s^*(\lambda)\|_{\omega}\leq C_2\mbox{ for all }\lambda\in \omega,\mbox{ that is, for }
\lambda=\alpha_i\mbox{ or }\beta_i.
$$
By symmetry, it suffices
to consider the case $\lambda=\alpha_i$ and $s\in S$. Write $s=f_J v f_J^{-1}$ where $|J|=3$ and $v\in S_{[3]}$. Let
$I=\{i\}$, $L=I\cup J$ and $l=|L|\leq 4$. As in (a) we have $s=f_L s' f_L^{-1}$ where $s'\in S_{[l]}$ and 
$\lambda=f_L^*(\alpha_j)$ for some $j\in [l]$. 

Then $s^*\lambda=(f_L s' f_L^{-1})^*f_L^*(\alpha_j)=f_L^*((s')^* \alpha_j)$. Since
we have boundedly many possibilities for $j$ and $s'$, we have $\|(s')^* \alpha_j\|_{\omega_l}\leq C_2$ for some absolute constant $C_2$, and as in the proof of (a) we have $\|f_L^*((s')^* \alpha_j)\|_{\omega}\leq \|(s')^* \alpha_j\|_{\omega_l}$. Thus, $\|s^*\lambda\|_{\omega}\leq C_2$, as desired.

\end{proof}

\begin{proof}[Proof of Theorem~\ref{thm:main_Torelli2}]
Recall that $S$ is the generating set for $\calI$ from Proposition~\ref{prop:Aestimate}. 
In the proof of Theorem~\ref{thm:main_Torelli} we showed that the Johnson kernel $\calK$ has a generating set consisting of elements of the form $x^y$ where $x=[s,t]$ for some $s,t\in S$ or $x\in S^{(4)}$ (in the notations from that proof) and $y=\prod\limits_{j=1}^{m}s_j$ where each $s_j\in S$ and $m\leq R(N+M)<2Rn^3$ where $R$, $N$ and $M$ are as in Theorem~\ref{thm:main_Torelli}(b). Also recall that every element of $S^{(4)}$
either lies in $S$ or has the form $s^{-1}t$ with $s,t\in S$.

Consider the following three words in the free group on two generators $a,b$: $w_1(a,b)=a$, $w_2(a,b)=a^{-1}b$
and $w_3(a,b)=[a,b]$. Let $C_1$ be the maximum of the constants $C_{w_1},C_{w_2},C_{w_3}$ from the conclusion
of Lemma~\ref{lem:expansion}(a). Now take any $x$ and $y$ as in the previous paragraph. 
By Lemma~\ref{lem:expansion}(a) we can write $x=\prod\limits_{i=1}^k T_{\gamma_i}$ where $\gamma_i\in \pi_1(\Sigma,p_0)$ is separating and $\|\gamma_i\|_{\omega}\leq C_1$. Hence $$x^y=\prod\limits_{i=1}^k T_{\gamma_i}^{\,\prod\limits_{j=1}^m s_j}=
\prod\limits_{i=1}^k T_{(\prod\limits_{j=m}^1 s_j^{-1})^*(\gamma_i)}.$$

By Lemma~\ref{lem:expansion}(b) for each $i$ we have 
$\|(\prod\limits_{j=m}^1 s_j^{-1})^*(\gamma_i)\|_{\omega}\leq C_2^{m}\|\gamma_i\|_{\omega}\leq 
C_2^{2Rn^3} C_1\leq (C_1 C_2)^{2Rn^3}$.
Thus, Theorem~\ref{thm:main_Torelli2} holds with $D=(C_1 C_2)^{2R}$.
\end{proof}

\end{document}